\documentclass[reqno,11pt]{amsart}

\usepackage{amsmath,amssymb,amsthm,bbm,amsxtra}
\usepackage[ansinew]{inputenc} % input encoding
\DeclareMathAlphabet{\mathpzc}{OT1}{pzc}{m}{it}

\usepackage{color}
\usepackage{verbatim}
\usepackage{enumitem}
\usepackage{csquotes}
\usepackage[english]{babel}
\usepackage[linktocpage]{hyperref}

\usepackage{mathrsfs}
\usepackage{stmaryrd}
\usepackage{tikz}
\usetikzlibrary{arrows,matrix,positioning}
\usepackage{enumitem}
\usepackage{bm}

\usepackage[linktocpage]{hyperref}
\usepackage{cleveref}

\usepackage{appendix}

\usepackage[colorinlistoftodos]{todonotes}

\usepackage[frak=esstix]{mathalfa}

\usepackage[backend=biber,style=alphabetic,maxbibnames=99]{biblatex}
\makeatletter
\def\@tocline#1#2#3#4#5#6#7{\relax
 \ifnum #1>\c@tocdepth % then omit
 \else
  \par \addpenalty\@secpenalty\addvspace{#2}%
  \begingroup \hyphenpenalty\@M
  \@ifempty{#4}{%
   \@tempdima\csname r@tocindent\number#1\endcsname\relax
  }{%
   \@tempdima#4\relax
  }%
  \parindent\z@ \leftskip#3\relax \advance\leftskip\@tempdima\relax
  \rightskip\@pnumwidth plus4em \parfillskip-\@pnumwidth
  #5\leavevmode\hskip-\@tempdima
   \ifcase #1
    \or\or \hskip 1em \or \hskip 2em \else \hskip 3em \fi%
   #6\nobreak\relax
  \dotfill\hbox to\@pnumwidth{\@tocpagenum{#7}}\par
  \nobreak
  \endgroup
 \fi}
\makeatother

\makeatletter
\def\@seccntformat#1{%
  \protect\textup{\protect\@secnumfont
    \ifnum\pdfstrcmp{subsection}{#1}=0 \bfseries\fi% subsection # in \bfseries
    \csname the#1\endcsname
    \protect\@secnumpunct
  }%
}  
\makeatother

\theoremstyle{plain} % default
\newtheorem{theorem}{Theorem} % [section]
\newtheorem{lemma}[theorem]{Lemma}

\newtheorem{corollary}[theorem]{Corollary}
\theoremstyle{definition}

\theoremstyle{remark}

\numberwithin{equation}{section}
\numberwithin{theorem}{section}

\newcommand{\N}{\mathbb{N}}  % natural numbers
\newcommand{\Z}{\mathbb{Z}}  % integers
\newcommand{\R}{\mathbb{R}}  % real numbers
\newcommand{\SL}{\mathrm{SL}}

\newcommand{\SO}{\mathrm{SO}}

\renewcommand{\d}{\mathrm{d}}

\newcommand{\iu}{\mathrm{i}}   % imaginary unit
\newcommand{\e}{\mathrm{e}}    % Euler's number

\newcommand{\8}{\kern-10pt}

\newcommand{\na}{\,\, {\raise.4pt\hbox{$\shortmid$}}{\hskip-2.0pt\to}\, \, }

\newcommand{\4}{\kern1pt}

\newcommand{\ffrac}[2]{\raise.5pt\hbox{\small$\4\displaystyle\frac{\,#1\,}{\,#2\,}\4$}}

\newcommand{\be}{\begin{equation}}
\newcommand{\ee}{\end{equation}}

\begin{document}

\title[Pair Correlation]{On the $\sigma$-Pair Correlation Density of Quadratic Sequences Modulo One}
\author{}
\subjclass[2000]{}


\date{March 11, 2022}

\author{Thomas Hille}
\address{Mathematics department, Northwestern University, Chicago, IL}

\email{thomas.hille@northwestern.edu}

\begin{abstract} 
In this note we study the $\sigma$-pair correlation density
\begin{equation*}
\mathrm R_2^\sigma([a,b], \{ \theta_n \}_n, N)= \frac{1}{N^{2-\sigma}} \# \big \{ 1 \leq j \neq k \leq N \, \big| \, \theta_{j} - \theta_{k} \in \big [ \frac{a}{N^\sigma},\frac{b}{N^\sigma}  \big ]+ \Z \big \}
\end{equation*}
of a sequence $\{ \theta_n\}_n$ that is equidistributed modulo one for $0 \leq \sigma <2$. 

The case $\sigma=1$ is commonly referred to as the \emph{pair correlation density} and the sequence $\{ n^2 \alpha \}_n$ has been of special interest due to its connection to a conjecture of Berry and Tabor on the energy levels of generic completely integrable systems.

We prove that if $\alpha$ is Diophantine of type $3-\epsilon$ for every $\epsilon>0$, then for any $0 \leq \sigma <1$
\begin{align*}
\mathrm R_2^\sigma([a,b], \{ \alpha n^2 \}_n, N) \to b-a, \text{ as } N \to \infty.
\end{align*}
In this case, we say that the sequence exhibits $\sigma$-pair correlation. 

In addition to this, we show that for any $0 \leq \sigma < \frac{1}{4}(9 -\sqrt{17})=1.21922...$ there is a set of full Lebesgue measure such that the sequence $\{ \alpha n^2 \}_n$ exhibits $\sigma$-pair correlation.

\end{abstract}
\maketitle

\vspace{-1.5em}

%%%% Table of Contents
\tableofcontents

\vspace{-2.5em}

%%%%%%%%%%%%%%%%%%%%%%%%%%%%%%%%%%%%%%%%%%%%%%%%%%%%%%%%%%%%%%%%%%%%%%%%%%%%%%%%%%%%%%%%%%%%%%%%%%%%%%%%%%%%%%%%%%%%%%%%
%%%%%%%%%%%%%%%%%%%%%%%%%%%%%%%%%%%%%%%%%%%%%%%%%%%%%%%%%%%%%%%%%%%%%%%%%%%%%%%%%%%%%%%%%%%%%%%%%%%%%%%%%%%%%%%%%%%%%%%%

%%%%%%%%%%%%%%  Section 1: Introduction

%%%%%%%%%%%%%%%%%%%%%%%%%%%%%%%%%%%%%%%%%%%%%%%%%%%%%%%%%%%%%%%%%%%%%%%%%%%%%%%%%%%%%%%%%%%%%%%%%%%%%%%%%%%%%%%%%%%%%%%%
%%%%%%%%%%%%%%%%%%%%%%%%%%%%%%%%%%%%%%%%%%%%%%%%%%%%%%%%%%%%%%%%%%%%%%%%%%%%%%%%%%%%%%%%%%%%%%%%%%%%%%%%%%%%%%%%%%%%%%%%

\section{Introduction}
\label{section:Introduction}

\subsection{}The \emph{pair correlation density} for a sequence of $N$ numbers $\theta_1, \dots, \theta_N$ which are equidistributed modulo one as $N$ tends to infinity, measures the distribution of spacings between the elements $\theta_n$ of the sequence at distances of order of the mean spacing $N^{-1}$. More precisely, the \emph{pair correlation function} is defined as follows
\begin{equation*}
\mathrm R_2([a,b], \{ \theta_{n}\}_n, N) = \frac{1}{N} \# \bigg \{ 1 \leq j \neq k \leq N \, \bigg| \, \theta_{j} - \theta_{k} \in \bigg [ \frac{a}{N},\frac{b}{N}  \bigg ]+ \Z \bigg \},
\end{equation*}
for any interval $[a,b]$. 

The pair correlation function $\mathrm R_2(\, \cdot \,, \{ \theta_n \}_n, N)$ is not a probability measure in general. However, if the sequence $\theta_n$ arises from independent and uniformly distributed random variables on $[0,1]$, then it is well-known that 
\begin{equation}
\label{Poissonian pair correlation}
\mathrm R_2([a,b],\{ \theta_{n}\}_n,N) \to b-a, \text{ as } N \to \infty.
\end{equation}
Thus, we will say that a deterministic sequence $\{ \theta_n\}_n$ exhibits \emph{Poissonian pair correlation} if \eqref{Poissonian pair correlation} holds for all intervals $[a,b]$. 

At this point let us also remark the connection between a sequence exhibiting Poissonian pair correlation and equidistribution modulo one: the former condition implies the latter (see \cite{grepstad-larcher:2017,aistleitner-lachmann-pausinger:2018}) but not vice-versa.

\subsection{} In 1916, Weyl \cite{weyl:1916} proved that the sequence  of fractional parts $\{ \alpha n^d\}_{n}$ is equidistributed in $[0,1]$. A finer question in the context of the pseudo-randomness of such a sequence is whether the corresponding spacings behave like those of independent and uniformly distributed random variables on $[0,1]$. 

For $d=1$ the answer is no and in fact the consecutive spacings have at most three values; this result is commonly referred to as the Steinhaus conjecture or the three distance theorem (see \cite{sos:1957,swierczkowski:1959,marklof-strombergsson:2017}). 

For $d \geq 2$, this is believed to be true for the \emph{pair correlation statistics} depending only on the Diophantine properties of $\alpha$: We say that a real number $\alpha$ is \emph{Diophantine of type $\kappa$} if there is a constant $c >0$ such that 
\begin{equation*}
\bigg | \alpha - \frac{p}{q} \bigg| \geq \frac{c}{q^\kappa}
\end{equation*}
for all $p \in \Z$ and $q \in \N_{>0}$. We say that $\alpha$ is \emph{Diophantine} if it is Diophantine of type $2+\epsilon$ for any $\epsilon>0$. For reference, a rational number $\alpha$ is of type $\kappa =1$; an irrational number $\alpha$ is of type $\kappa \geq 2$; an algebraic irrational number $\alpha$, due to Roth's theorem, is of type $\kappa = 2 +\epsilon$ for all $\epsilon>0$. 

Rudnick and Sarnak observed (see \cite{rudnick-sarnak:1998} Remark 1.2 or \cite{heath-brown:2010}) that if $\alpha$ is not Diophantine of type $d+1$, then $\{\alpha n^d\}_n$ cannot exhibit Poissonian pair correlation \eqref{Poissonian pair correlation}. In the special case $d=2$, it seems likely to be true that $\{ \alpha n^2\}_n$ exhibits Poissonian pair correlation if $\alpha$ is Diophantine of type $3-\epsilon$ for every $\epsilon>0$.

However, one can find a set $P \subset \R$ of full Lebesgue measure such that for any $\alpha \in P$, the sequence $\{ \alpha n^d \}_n$ exhibits Poissonian pair correlation \eqref{Poissonian pair correlation}. Results of this nature are commonly referred to as \emph{metric Poisson pair correlation}. This was first proved by Rudnick and Sarnak \cite{rudnick-sarnak:1998}, but let us also mention proofs using different methods in the special case $d=2$ by Marklof and Str\"ombergsson \cite{marklof-stroembergsson:2003} as well as Heath-Brown \cite{heath-brown:2010}. 

Yet, as of this moment there is no specific $\alpha$ for which the Poissonian pair correlation property \eqref{Poissonian pair correlation} has been verified. The most compelling result so far for an explicit $\alpha$ was proved by Heath-Brown \cite{heath-brown:2010}, who showed that for a number $\alpha$ of type $\kappa = \frac{9}{4}$ the sequence $\{ \alpha n^2 \}_n$ satisfies
\begin{equation}
\label{heath-brown}
\mathrm R_2(L[-1,1],\{\alpha n^2\}_n,N)= 2L + \mathcal O(L^{\frac{7}{8}}) + \mathcal O((\log N)^{-1}),
\end{equation}
for $1 \leq L \leq \log N$ as $N$ and $L$ tend to infinity. 

In light of our result below, it seems plausible to expect $\{ \alpha n^2 \}_{n}$ to exhibit Poissonian pair correlation if $\alpha$ is Diophantine of type $3-\epsilon$ for any $\epsilon>0$. In fact, this would follow from a conjecture due to Truelsen \cite{truelsen:2010} on the average value of  $\tau_{M,N}(n) = \# \{ (a,b) \in \N^2 \, | \, a \leq M, \, b \leq N, \, ab =n \}$
in short arithmetic progressions.

\subsection{} Heath-Brown's result \eqref{heath-brown} motivates the following notion, already introduced by Nair and Pollicott \cite{nair-pollicott:2007}. Namely, for $0 \leq \sigma <2$ let us set
\begin{equation*}
\mathrm R_2^\sigma([a,b], \{ \theta_{n}\}_n, N) = \frac{1}{N^{2-\sigma}} \# \bigg \{ 1 \leq j \neq k \leq N \, \bigg| \, \theta_{j} - \theta_{k} \in \bigg [ \frac{a}{N^\sigma},\frac{b}{N^\sigma}  \bigg ]+ \Z \bigg \}.
\end{equation*}
We say that a deterministic sequence $\{ \theta_n\}_n$ has \emph{$\sigma$-pair correlation} if 
\begin{equation}
\label{sigma pair correlation}
\mathrm R_2^\sigma([a,b],\{ \theta_{n}\},N) \to b-a, \text{ as } N \to \infty,
\end{equation}
 holds for all intervals $[a,b]$. The case $\sigma =0$ corresponds vaguely speaking to that of equidistribution on $[0,1]$ and the case $\sigma=1$ to that of Poissonian pair correlation. Moreover, note that for any two $0 \leq \sigma_1, \sigma_2 <2$
\begin{equation*}
\mathrm R_2^{\sigma_2}([a,b], \{ \theta_{n}\}_n, N) = \frac{1}{N^{\sigma_1-\sigma_2}} \mathrm R_2^{\sigma_1}(N^{\sigma_1-\sigma_2}[a,b] , \{ \theta_{n}\}, N) 
\end{equation*}
and thus if a sequence $\{\theta_n\}_n$ exhibits $\sigma$-pair correlation for some $\sigma$, then 
\begin{equation*}
 \mathrm R_2(N^{1-\sigma}[a,b] , \{ \theta_{n}\}_n, N) \sim N^{1-\sigma}(b-a), \text{ as } N \to \infty.
\end{equation*}

\noindent Analogously to the case $\sigma =1$, it is known that a sequence exhibiting $\sigma$-pair correlation for some $0<\sigma <1$ is equidistributed in $[0,1]$ (see \cite{steinerberger:2020}). 
 
 In this note we obtain for the sequence $\theta_n = \alpha n^2$ the following result
 \begin{theorem}
 \label{Theorem 1}
 Let $0 \leq \sigma <1$ and $2 \leq \kappa < 1+ \frac{2}{\sigma}$. Suppose $\alpha$ is Diophantine of type $\kappa$, then $\{ \alpha n^2 \}_n$ exhibits $\sigma$-pair correlation \eqref{sigma pair correlation}, that is
 \begin{equation}
\mathrm R_2^\sigma([a,b],\{\alpha n^2 \},N) \to b-a, \text{ as } N \to \infty,
\end{equation}
for all intervals $[a,b]$.
 \end{theorem}
 
\noindent Let us note that the same method can be applied to other sequences of interest like the fractional parts of $\{\frac{\alpha n^2}{N}\}_{1 \leq n \leq N}$ (see \cite{marklof:2010}). 
 
 In this context it is natural to understand the set of $0 \leq \sigma <2$ for which $\{ \theta_n \}_n$ exhibits $\sigma$-pair correlation. One can ask how large $\sigma$ can be taken as to still expect the sequence $\{\theta_n\}_n$ to exhibit $\sigma$-pair correlation. Thus, let us introduce the following quantity
 \begin{equation*}
 \sigma(\{\theta_n\}_n) := \sup \bigg \{0 \leq \sigma <2  \, \bigg | \,  \{ \theta_n \}_n \text{ exhibits } \sigma \text{-pair correlation} \bigg \}.
 \end{equation*}
 
\noindent With this notation, \Cref{Theorem 1} clearly implies
 
 \begin{corollary}
 Let $\alpha \in \R$ be Diophantine of type $3-\epsilon$ for any $\epsilon>0$, then $\sigma(\{\alpha n^2\}) \geq 1$.
 \end{corollary}
 
\noindent Moreover, for almost every $\alpha$ we obtain using the approach of Rudnick and Sarnak \cite{rudnick-sarnak:1998}
 
 \begin{theorem}
 \label{Theorem 2} For any $0 \leq \sigma <\frac{1}{4}(9 -\sqrt{17})=1.21922...$ there is a set $P \subset \R$ of full Lebesgue measure such that $\{\alpha n^2 \}_n$ exhibits $\sigma$-pair correlation for all $\alpha \in P$.
 \end{theorem}
 
 \noindent Similarly as above, this implies
 
 \begin{corollary}
 For almost all $\alpha \in \R$ we have $\sigma(\{\alpha n^2 \}_n) \geq  \frac{1}{4}(9 -\sqrt{17})$. 
 \end{corollary}
 
 More generally, for the sequence $\{ \alpha n^d\}_n$ with $d \geq 3$, we obtain in analogy to \Cref{Theorem 2}.
 
  \begin{theorem}
 \label{Theorem 3} For any $0 \leq \sigma <2+2^{-d} - \sqrt{1+4^{-d}}$ there is a set $P \subset \R$ of full Lebesgue measure such that $\{\alpha n^d \}_n$ exhibits $\sigma$-pair correlation for all $\alpha \in P$. Note that $2+2^{-d} - \sqrt{1+4^{-d}}>1$ for all $d \in \N$ and this expression converges to $1$ as $d$ tends to infinity.
 \end{theorem}
 
 The proof of \Cref{Theorem 3} goes along the exact same lines as that of \Cref{Theorem 2} with some minor modifications. A question that arises out of this study -- albeit beyond the scope of this note -- is whether there are numbers $\alpha$ for which the sequence $\{ \alpha n^d \}_n$ exhibits $\sigma$-pair correlation for some $\sigma \geq 2+2^{-d} - \sqrt{1+4^{-d}}$.

\subsection{} The strategy of the proof of \Cref{Theorem 1} is classical and relies on studying the frequency side of the counting problem 
\begin{equation*}
\frac{1}{N^{2-\sigma}} \# \bigg \{ 1 \leq j \neq k \leq N \, \bigg| \, \alpha(j^2 - k^2)\in \bigg [ \frac{a}{N^\sigma},\frac{b}{N^\sigma}  \bigg ]+ \Z \bigg \}.
\end{equation*}
As this counting problem lies on the torus, it translates on the frequency side to sums of \emph{theta sums}. For context, the classical problem of counting values of quadratic forms at integral points, which lies on the real line, translates on the frequency side to integrals of \emph{theta sums} (see \cite{birch-davenport:1958a,marklof:2003,goetze:2004,bghm:2022}).

We can view this problem as a counting problem regarding the distribution of values of the quadratic form $\begin{pmatrix} \alpha & 0 \\ 0 & -\alpha \end{pmatrix}$ modulo one at integral points away from the isotropic rational subspaces. 

From this point of view, two difficulties arise: On the one hand, the stabilizer of this form is a torus and thus we cannot exploit some standard techniques from homogeneous dynamics. On the other hand, the main difficulty in our case is the avoidance of isotropic rational subspaces. More precisely, the counting problem above leads us to study sums of the form
\begin{equation*}
\sum_{n \in \Z \setminus \{0 \}} \sum_{\substack{(x,y) \in \Z^2 \\ |x| \neq |y|}}\e^{\iu \Omega_{n}(x,y)},
\end{equation*}
where $\Omega_n$ is a quadratic form in the Siegel upper half-space of degree two such that the rational isotropic subspaces of the real part are precisely $\{(x,y) \in \Z^2 \, | \, x= \pm y \}$. In \Cref{Bound for Exp Sum} we show that such an expression can essentially be rewritten as the difference of two theta series. However, we are unable to exploit this difference, which is the likely culprit of the deficiency in \Cref{Theorem 1}. Nevertheless, building upon a method going back to the seminal work of G\"otze \cite{goetze:2004},  the bounds we obtain for both of these theta series are related to the number of Diophantine approximations $\alpha$ up to a certain height. Vaguely speaking, the proof of \Cref{Theorem 1} reduces to the following counting problem
\begin{equation*}
\frac{1}{N} \# \bigg \{ \bm v \in \begin{pmatrix}  N^{\frac{2+\sigma}{2}} &\\ &N^{-\frac{2+\sigma}{2}}  \end{pmatrix} \begin{pmatrix}  1 & \alpha\\ &1 \end{pmatrix} \Z^2 \, \bigg | \, \| \bm v \| \leq N^\frac{\sigma}{2} \bigg\},
\end{equation*}
which can be seen to be of order $\mathcal O( N^{\sigma-1})$ as long as $\alpha$ is Diophantine of type specified in \Cref{Theorem 1}.

\subsection{} Let us mention a few related results and further extensions regarding the pair correlation problem.

Zelditch \cite{zelditch:1998,zelditch:1998A} studied the pair correlation for sequences of the form $\alpha N \phi(\frac{n}{N}) + \beta n$, where $\phi$ is a fixed polynomial satisfying $\phi'' \neq 0$ on $[-1,1]$ and obtained metric results, similar to that of Sarnak and Rudnick \cite{rudnick-sarnak:1998}, with the caveat of dealing with the averaged pair correlation function. Marklof and Yesha \cite{marklof-yesha:2018} were able to prove and expand on Zelditch's averaged pair correlation for the special case $\frac{(n-\alpha)^2}{2N}$ for explicit $\alpha$ satisfying a Diophantine condition.

Boca and Zaharescu \cite{boca-zaharescu:2000} studied the pair correlation of sequences of the form $\{f(n) \mod p\}_{1 \leq n \leq N}$, where $f$ is a rational function with integer coefficients.

Rudnick and Zaharescu \cite{rudnick-zaharescu:1999} obtained metric results for sequences of the form $\{\alpha a(n)\}$, where $a(n)$ is a lacunary sequence. This motivated the study of the pair correlation problem from the point of view of the so-called additive energy (for a set $A$ of real numbers, the additive energy $E(A)$ is defined to be $\# \{(a,b,c,d) \in A^4 \, | \, a+b = c+d\}$), which has proven to yield fruitful results. Let us mention in this regard results of Aistleitner, Larcher and Lewko \cite{aistleitner-larcher-lewko:2017}, Bloom and Walker \cite{bloom-walker:2020}, Aistleitner, El-Baz and Munsch \cite{aistleitner-elbaz-munsch:2021}.

The pair correlation of the sequence $\{ \alpha n^\theta \}_n$ has attracted special attention. For $\theta = \frac{1}{2}$, Elkies and McMullen \cite{elkies-mcmullen:2004} showed, using techniques from homogeneous dynamics, that the gap distribution of $\{\sqrt n\}_n$ is not Poissonian. Surprisingly, after removing the perfect squares out of this sequence, El-Baz, Marklof and Vinogradov \cite{elbaz-marklof-vinogradov:2015} showed that the resulting sequence does exhibit Poissonian pair correlation. For $\theta \in (0, \frac{14}{41})$ and $\alpha >0$, Lutsko, Sourmelidis and Technau \cite{lutsko-athanasios-technau:2021} show that the sequence $\{ \alpha n^\theta \}_n$ exhibits Poissonian pair correlation without any further restrictions on $\alpha$. In particular their result gives an example of a sequence that exhibits Poissonian pair correlation, but not triple pair correlation.

\subsection{} Finally let us elaborate more on the underlying motivation to this problem. The special case $d=2$ is of particular interest due to its connection between number theory and theoretical physics. Namely, the distribution of spacings of the sequence $\{\alpha n^2\}_n$ modulo one are related to the spacings between the energy levels of the \emph{boxed oscillator} \cite{berry-tabor:1977}, a particle in a two-dimensional potential well with hard walls in one direction and harmonic binding in the other. 

For a sequence of numbers $\theta_1, \dots, \theta_N$ modulo one, denote the corresponding \emph{order statistics} on $[0,1]$ by $\theta_{(1)} \leq \dots \leq \theta_{(N)}$. The corresponding \emph{spacing measure} is given by
\begin{equation*}
\mu_2(\{ \theta_n \}_n,N) := \frac{1}{N} \sum_{n =1} ^N \delta_{N(\theta_{(j+1)}-\theta_{(j)})},
\end{equation*}
where we set $\theta_{(N+j)}= \theta_{(j)}$ and $\delta_x$ denotes the Dirac measure at $x$. If the sequence $\theta_n$ arises from a sequence of independent and uniformly distributed random variables on $[0,1]$, then it is well-known that 
\begin{equation}
\label{poisson}
\mu_2(\{\theta_n\}_n,N) \to \e^{-x} \d x, \text{ as } N \to \infty.
\end{equation}
We say that a deterministic sequence $\{\theta_n\}_n$ is \emph{Poissonian} if \eqref{poisson} holds. 

The standard approach to the analysis of spacings is through \emph{higher correlation functions} defined analogously as before by
\begin{equation*}
\mathrm R_k(\prod_{i=1}^{k-1}[a_i,b_i],\{ \theta_{n}\}_n,N)= \frac{1}{N} \# \bigg \{\bm x \in [N]^k \, \bigg | \, \theta_{x_i}-\theta_{x_{i+1}} \in [a_i,b_i] \text{ for all } i \bigg\},
\end{equation*} 
where $[N]^k$ denotes the set of all $k$-tuples $\bm x = (x_1,\dots, x_k)$ of \emph{distinct} integers in $\{1,\dots, N\}$.
One can show that if all correlation functions are \emph{Poissonian}, in the sense that
\begin{equation*}
\mathrm R_k(\prod_{i=1}^{k-1}[a_i,b_i],\{ \theta_{n}\}_n,N) \to \prod_{i=1}^{k-1} (b_i-a_i), \text{ as } N \to \infty,
\end{equation*}
then $\{\theta_n\}_n$ is Poissonian. 

In this generality, Rudnick, Sarnak and Zaharescu \cite{rudnick-sarnak-zaharescu:2001} proved that if $\alpha$ is \emph{not Diophantine of type $3$} and the denominators of the corresponding rational approximations (for which the Diophantine property fails) are essentially square free, then $\{\alpha n^2\}_n$ is \emph{Poissonian along a subsequence $N_j$}. Of course, as mentioned above, such a sequence cannot possibly be Poissonian along the entire sequence $N \in \N$; In fact, if $\alpha$ is not Diophantine of type $3$, then there is a subsequence $N_l$ for which $\mu_2(\{\alpha n^2 \}_n,N_l)$ converges to a measure supported on $\N_0$. Nevertheless, their analysis leads them to conjecture that $\{ \alpha n^2 \}_n$ is \emph{Poissonian} if $\alpha$ is Diophantine of type $2+\epsilon$ for every $\epsilon>0$ and the denominators of the convergents to $\alpha$ are essentially square free (see \cite{rudnick-sarnak-zaharescu:2001} for details).

\hspace{1cm}

\emph{Acknowledgments.} I thank I. Khayutin for very helpful discussions and G. Margulis for introducing me to many of the ideas contained in this note.

\section{The Pair Correlation Functional and Test Functions}
\label{section:Test Functions}

\subsection{}
\label{subsection:Functional} Let $\{ \theta_n \}_n$ be a sequence and $0 \leq \sigma <2$. For sufficiently fast decaying functions $f: \R \to \R$ and $\psi: \R^2 \to \R$ we define the \emph{$\sigma$-pair correlation functional} to be
\begin{equation*}
\mathrm R_2^{\sigma}(f,\psi, \{ \theta_{n} \}_n,N):= \frac{1}{N^{2-\sigma}} \sum_{\substack{j,k \in \Z \\ \, |j| \neq |k|}}\psi \bigg(\frac{1}{N} \begin{pmatrix} j \\ k \end{pmatrix}\bigg) \sum_{n \in \Z} f(N^{\sigma}(\theta_{j}-\theta_{k}+n)).
\end{equation*}
Note that with this definition we have
\begin{equation*}
\mathrm R_2^\sigma([a,b],\{ \theta_{n} \}_n,N)= 4\mathrm R_2^{\sigma}(\mathbbm 1_{[a,b]},\mathbbm 1_{[-1,1]^2}, \{ \theta_{n} \}_n,N),
\end{equation*}
for any interval $[a,b]$.

\subsection{} We shall approximate the indicator function corresponding to the interval $[a,b]$ and the square $[-1,1]^2$ from above and from below by a class of functions that are well-suited to this problem. More precisely, we will call a pair of functions $f:\R \to \R, \psi: \R^2 \to \R$ \emph{test functions of class $\mathcal P$} if 
\begin{align}
\begin{aligned}
\label{def:class P} \widehat f \in C_c(\R)  \text{ or } \widehat f(t)=\e^{-2 \pi |t|} \text{ and } \widehat \psi \in C_c(\R^2).
\end{aligned}
\end{align}
The following approximation lemma by functions whose Fourier transform is compactly supported is well-known and modified slightly from \cite{marklof:2003} (see \S 8.6.2-8.6.4). We provide a proof for completeness.

\begin{lemma}
\label{Approximation Lemma}
Let $I \subset \R$ be a closed (finite) interval, $g: \R \to \R_{\geq 0}$ a non-negative and continuous function and $\epsilon >0$. There exist functions $h_-$ and $h_+$ such that
\begin{align}
\begin{aligned}
& \widehat{h_\pm} \in C_c(\R),\\
&-\frac{\epsilon}{\pi (1+s^2)} \leq h_-(s) \leq g(s) \mathbbm 1_{I}(s) \leq h_+(s), \; \text{ for all } s \in \R,\\
&\int_I g(s) \, \d s \leq \widehat{h_{+}}(0)  \leq \int_I g(s) \, \d s +  \epsilon, \\
&\int_I g(s) \, \d s - \epsilon \leq \widehat{h_{-}}(0) \leq \int_I g(s) \, \d s.
\end{aligned}
\end{align}
\end{lemma}

\begin{proof}
Let $\chi_{\pm} \in C_c^\infty(\R)$ be compactly supported and continuous functions such that
\begin{align*}
&0\leq \chi_- \leq g \cdot \mathbbm 1_I \leq \chi_+, \\
& \int_{\R} (\chi_+-\chi_-) \d s <\epsilon.
\end{align*}
Let us set
\begin{align*}
h_{\pm,\epsilon}(s) = \chi_{\pm}(s) \pm \frac{\epsilon}{\pi(1+s^2)}.
\end{align*}
Then,
\begin{align*}
&0 \leq h_{-,\epsilon}(s) + \frac{\epsilon}{\pi(1+s^2)} \leq g(s) \leq  h_{+,\epsilon}(s) -\frac{\epsilon}{\pi(1+s^2)},\\
& \widehat{h_{\pm,\epsilon}}(u) =\widehat{\chi_{\pm}}(u)\pm \epsilon \,\e^{-2 \pi |u|}.
\end{align*}
Next, fix a continuous and compactly supported function $\chi \in C_c(\R)$ such that
\begin{align*}
& \text{supp}( \chi) \subseteq [-2,2], \\
& 0 \leq \chi \leq 1, \\
& \chi |_{[-1,1]} \equiv 1.
\end{align*}
For $P \geq 1$, to be determined below, set
\begin{align*}
\widehat {h_{\pm,\epsilon,P}}(u) = \widehat{h_{\pm,\epsilon}}(u) \chi(\frac{u}{P})
\end{align*}
and observe that $\widehat{h_{\pm,\epsilon,P}} \in C_c(\R)$ is continuous and compactly supported with $\text{supp}( \widehat{h_{\pm,\epsilon,P}}) \subset [-2P,2P]$.
We claim that we can choose $P \geq 1$, depending on $\epsilon$ and $\chi_\pm$ only, such that \begin{align*}
|h_{\pm,\epsilon,P}(s) - h_{\pm,\epsilon}(s) | \leq \frac{\epsilon}{\pi(1+s^2)}, \; \text{ for any }s \in \R.
\end{align*}
Indeed, for some $C$ depending on $\epsilon$ and $\chi_{\pm}$ we have on the one hand
\begin{align*}
|h_{\pm,\epsilon,P}(s) - h_{\pm,\epsilon}(s) | &\leq \int_{\R} |\widehat{h_{\pm,\epsilon}}(u)|\bigg(1-\chi \big(\frac{u}{P}\big) \bigg) \d u \\ & \leq \int_{|u| \geq P}  |\widehat{h_{\pm,\epsilon}}(u)| \, \d u \leq \frac{C}{P},
\end{align*}
on the other hand, after integrating by parts twice,
\begin{align*}
|h_{\pm,\epsilon,P}(s) &- h_{\pm,\epsilon}(s) | \\ & \leq \frac{1}{4 \pi^2 s^2} \bigg( \int_{|u| \geq P} |\widehat{h_{\pm,\epsilon}}''(u)|\d u \\ & \qquad \qquad \qquad + \frac{2}{P} \int_{\R} |\widehat{h_{\pm,\epsilon}}'(u) \chi \big(\frac{u}{P} \big)|\d u \\ & \qquad \qquad \qquad \qquad  \quad  + \frac{1}{P^2} \int_{\R} |\widehat{h_{\pm,\epsilon}}(u) \chi''\big(\frac{u}{P}\big)|\d u \bigg) \\&\qquad \qquad \qquad \qquad \qquad  \qquad \qquad  \qquad \qquad \qquad \quad  \leq \frac{C}{Ps^2}
\end{align*}
and thus
\begin{align*}
|h_{\pm,\epsilon,P}(s) - h_{\pm,\epsilon}(s) | \leq \frac{C}{P} \min \bigg\{ 1, \frac{1}{s^2} \bigg\} \leq \frac{\epsilon}{\pi(1+s^2)},
\end{align*}
whenever $P \geq 2 C \pi \epsilon^{-1}$. Hence, we conclude that
\begin{align*}
 g(s)\mathbbm 1_{I}(s)& \geq h_{-,\epsilon,P}(s)  \geq g(s)\mathbbm 1_{I}(s)-(\chi_+(s)-\chi_-(s))-2\frac{\epsilon}{\pi(1+s^2)}  \\
 g(s)\mathbbm 1_{I}(s) & \leq h_{+,\epsilon,P}(s) \leq g(s)\mathbbm 1_{I}(s) + (\chi_{+}(s) -\chi_-(s))+ 2\frac{\epsilon}{\pi(1+s^2)},
\end{align*}
for all $s \in \R$, which implies
\begin{align*}
&\widehat{g \cdot \mathbbm 1_{I}}(0) \leq \int_{\R} h_{+,\epsilon,P}(s) \, \d s  \leq \widehat {g\cdot \mathbbm 1_{I}}(0) + 3 \epsilon, \\
&\widehat{g \cdot \mathbbm 1_{I}}(0)-3 \epsilon \leq \int_{\R} h_{-,\epsilon,P}(s) \, \d s \leq \widehat{g \cdot \mathbbm 1_{I}}(0).
\end{align*}
Finally, also note by construction that
\begin{align*}
h_{-,\epsilon,P}\geq -\frac{\epsilon}{\pi(1+s^2)} +h_{-,\epsilon} = \chi_- -2\frac{\epsilon}{\pi(1+s^2)} \geq -2\frac{\epsilon}{\pi(1+s^2)},
\end{align*}
where we use that $\chi_-$ is non-negative.
\end{proof}

\subsection{} The sufficiency of the class of test functions introduced above is established by the following
\begin{lemma}
\label{sufficient test functions}
Suppose that for any two test functions $f:\R \to \R$ and $\psi: \R^2 \to \R$ of class $\mathcal P$ (see \eqref{def:class P}) the following identity holds
\begin{align*}
\lim_{N \to \infty} \mathrm R_2^{\sigma}(f,\psi \cdot \e^{-\| \, \cdot \, \|^2},\{\theta_{j}\}_j,N) = \widehat f(0) \big(\mathcal F( {\psi \cdot \e^{-\| \, \cdot \, \|^2}})\big)(\bm 0),
\end{align*}
where $\cdot$ denotes pointwise multiplication of two functions.
Then, for any $a<b$
\begin{align*}
\lim_{N \to \infty} \mathrm R_2^\sigma([a,b],\{\theta_{j}\}_j,N) =b-a.
\end{align*}
\end{lemma}

\begin{proof}
Let $\epsilon>0$. According to \Cref{Approximation Lemma} we can find $f_\pm:\R \to \R$ such that
\begin{align}
\begin{aligned}
\label{eq:215}
& \widehat{f_\pm} \in C_c(\R),\\
& -\frac{\epsilon}{\pi (1+s^2)} \leq f_-(s) \leq \mathbbm 1_{[a,b]}(s) \leq f_+(s) , \; \text{ for all } s \in \R, \\
&b-a \leq  \widehat {f_{+}}(0)  \leq b-a +  \epsilon, \\
&b-a - \epsilon \leq  \widehat {f_{-}}(0)  \leq b-a.
\end{aligned}
\end{align}
as well as $\psi_{\pm}: \R^2 \to \R$ satisfying
\begin{align}
\begin{aligned}
\label{eq:216}
& \widehat{\psi_\pm} \in C_c(\R^2),\\
&\psi_-(\bm v) \leq \mathbbm 1_{[-1,1]^2}(\bm v)\,\e^{\|\bm v\|^2} \leq \psi_+(\bm v), \; \text{ for all } \bm v \in \R^2 \\
&\int_{[-1,1]^2} \e^{\|\bm v\|^2} \d \bm v \leq \widehat{\psi_{+}}(\bm 0)  \leq \int_{[-1,1]^2} \e^{\|\bm v\|^2} \d \bm v +  \epsilon, \\
&\int_{[-1,1]^2} \e^{\|\bm v\|^2} \d \bm v - \epsilon \leq \widehat{\psi_{-}}(\bm 0) \leq \int_{[-1,1]^2} \e^{\|\bm v\|^2} \d \bm v.
\end{aligned}
\end{align}
By assumption, there is $N_0 \in \N$ such that for any $N \geq N_0$
\begin{align}
\label{eq:217}
&\big|R_2^\sigma(f_{\pm},\psi_\pm \cdot \e^{-\| \, \cdot \, \|^2}, \{\theta_{j}\}_j,N) - \widehat f_\pm(0) \big(\mathcal F(\psi_\pm \cdot \e^{-\| \, \cdot \, \|^2}) \big)(\bm 0) \big| <\epsilon, \\
\label{eq:218}
& \big|R_2^\sigma(\frac{2 \epsilon}{\pi (1+(\, \cdot \,)^2)},\psi_\pm \cdot \e^{-\| \, \cdot \, \|^2}, \{\theta_{j}\}_j,N) - 2 \epsilon \, \mathcal F(\psi_\pm \cdot \e^{-\| \, \cdot \, \|^2})(\bm 0) \big| <\epsilon,
\end{align}
where we use that
\begin{align*}
\int_{\R} \frac{\e^{-2 \pi \iu u s} \d s}{\pi(1+s^2)} = \e^{-2 \pi |u|}.
\end{align*}
Let $N \geq N_0$. Using \eqref{eq:215}, \eqref{eq:216}, \eqref{eq:217} and \eqref{eq:218}, we obtain the upper bound
\begin{align*}
&4\mathrm R_2^\sigma({[a,b],\{\theta_{j}\}_j,N}) - 4(b-a) 
\\ & \quad \leq \mathrm R_2^{\sigma}(f_{+},\psi_+ \cdot \e^{-\| \, \cdot \, \|^2}, \{\theta_{j}\}_j,N) - \widehat{f_-}(0)\big(\mathcal F(\psi_- \cdot \e^{-\| \, \cdot \, \|^2}) \big)(\bm 0)
\\ &  \quad  \quad  \leq \epsilon +  \widehat{f_+}(0)\big(\mathcal F(\psi_+ \cdot \e^{-\| \, \cdot \, \|^2}) \big)(\bm 0)- \widehat{f_-}(0)\big(\mathcal F(\psi_- \cdot \e^{-\| \, \cdot \, \|^2}) \big)(\bm 0)
\\ &  \quad  \quad  \quad =  \epsilon +  (\widehat{f_+}(0)-\widehat{f_-}(0))\big(\mathcal F(\psi_+ \cdot \e^{-\| \, \cdot \, \|^2}) \big)(\bm 0)
\\ &\quad  \quad  \quad\quad  \quad  \quad+\widehat{f_-}(0)\int_{\R^2}(\psi_+(\bm v)-\psi_-(\bm v)) \e^{-\|\bm v\|^2} \, \d \bm v 
\\ 
&  \quad \quad \quad  \quad  \quad  \leq \epsilon + 2 \epsilon(12+\epsilon)+2\epsilon(b-a).
\end{align*}
The corresponding lower bound is more delicate as $f_-$ is not everywhere positive. However, due to \Cref{Approximation Lemma}, $f_-+ \frac{2 \epsilon}{\pi (1+(\, \cdot \,)^2)}\geq 0$ as well as  \eqref{eq:215}, \eqref{eq:216}, \eqref{eq:217} and \eqref{eq:218}, we obtain
\begin{align*}
&4\mathrm R_2^\sigma({[a,b],\{\theta_{j}\}_j,N}) - 4(b-a) 
%\\ & \quad \geq \mathrm R_2^{\sigma}(f_{-},\mathbbm 1_{[-1,1]^2}, \{\theta_{j}\}_j,N) - \widehat{f_+}(0)\big(\mathcal F(\psi_+ \cdot \e^{-\| \, \cdot \, \|^2}) \big)(\bm 0)
%\\ &  \quad =\mathrm R_2^{\sigma}(f_{-}+\frac{2 \epsilon}{\pi (1+(\, \cdot \,)^2)},\mathbbm 1_{[-1,1]^2}, \{\theta_{j}\}_j,N)
% - \widehat{f_+}(0)\widehat{\psi_+}(\bm 0)  
  \\ & \quad\quad\quad\quad-\mathrm R_2^{\sigma}(\frac{2 \epsilon}{\pi (1+(\, \cdot \,)^2)},\mathbbm 1_{[-1,1]^2}, \{\theta_{j}\}_j,N)
\\ & \quad \geq  \mathrm R_2^{\sigma}(f_{-},\psi_- \cdot \e^{-\| \, \cdot \, \|^2}, \{\theta_{j}\}_j,N) - \widehat{f_+}(0)\big(\mathcal F(\psi_+ \cdot \e^{-\| \, \cdot \, \|^2}) \big)(\bm 0) \\
& \quad\quad\quad\quad-\mathrm R_2^{\sigma}(\frac{2 \epsilon}{\pi (1+(\, \cdot \,)^2)},(\psi_+-\psi_-) \cdot \e^{-\| \, \cdot \, \|^2}, \{\theta_{j}\}_j,N) 
\\&\quad  \geq -\epsilon +\widehat{f_-}(0)\big(\mathcal F(\psi_- \cdot \e^{-\| \, \cdot \, \|^2}) \big)(\bm 0) - \widehat{f_+}(0)\big(\mathcal F(\psi_+ \cdot \e^{-\| \, \cdot \, \|^2}) \big)(\bm 0)
\\ &  \quad\quad\quad\quad-\mathrm R_2^{\sigma}(\frac{2 \epsilon}{\pi (1+(\, \cdot \,)^2)},(\psi_+-\psi_-) \cdot \e^{-\| \, \cdot \, \|^2}, \{\theta_{j}\}_j,N)
%\\& \quad = -\epsilon  -( \widehat{f_+}(0)- \widehat{f_-}(0))\big(\mathcal F(\psi_+ \cdot \e^{-\| \, \cdot \, \|^2}) \big)(\bm 0) 
\\ &   \quad\quad\quad\quad - \widehat{f_-}(0)\int_{\R^2} (\psi_+(\bm v)-\psi_-(\bm v))\e^{-\|\bm v\|^2} \d \bm v 
\\ &  \quad  \quad  \quad  \quad  \quad  \quad  -\mathrm R_2^{\sigma}(\frac{2 \epsilon}{\pi (1+(\, \cdot \,)^2)},(\psi_+-\psi_-)\cdot \e^{-\| \, \cdot \, \|^2}, \{\theta_{j}\}_j,N) 
\\& \quad \geq -2\epsilon  -( \widehat{f_+}(0)- \widehat{f_-}(0))\big(\mathcal F(\psi_+ \cdot \e^{-\| \, \cdot \, \|^2}) \big)(\bm 0) 
\\ &   \quad\quad\quad\quad - \widehat{f_-}(0)\int_{\R^2} (\psi_+(\bm v)-\psi_-(\bm v))\e^{-\|\bm v\|^2} \d \bm v 
%\\ &  \quad  \quad  \quad  \quad  \quad  \quad -2 \epsilon \int_{\R^2} (\psi_+(\bm v)-\psi_-(\bm v))\e^{-\|\bm v\|^2} \d \bm v
\\& \quad \geq -\epsilon  -2 \epsilon(12+\epsilon)- 2\epsilon(b-a) -4\epsilon^2 .
\end{align*}
\end{proof}

\subsection{} The following lemma allows us to establish the connection to the Geometry of Numbers in the next section. In fact, it shows that the main error term is roughly speaking an \emph{incomplete theta sum}.

\begin{lemma}
\label{error term lemma}
Let $f:\R \to \R$ and $\psi: \R^2 \to \R$ be test functions of class $\mathcal P$ (see \eqref{def:class P}). Then,
\begin{align*}
&\mathrm R_2^{\sigma}(f, \psi \cdot \e^{-\| \, \cdot \, \|^2}, \{ \theta_{j} \}_j, N) \\ & \quad = \widehat f(0) \mathcal F(\psi \cdot \e^{-\| \, \cdot \,\|^2})(\bm 0) 
\\& \qquad \quad +\widehat f(0)\bigg( \frac{1}{N^2} \sum_{\substack{j,k \in \Z\\ |j| \neq |k|}} (\psi\cdot \e^{-\| \, \cdot \, \|^2}) \bigg( \frac{1}{N} \begin{pmatrix} j \\ k \end{pmatrix} \bigg) - \widehat f(0) \mathcal F(\psi \cdot \e^{-\| \, \cdot \,\|^2})(\bm 0) \bigg)
\\ & \qquad \qquad \quad +\frac{1}{N^2}\int_{\R^2} \widehat \psi(\bm \xi) \sum_{n \neq 0} \widehat f\bigg(\frac{n}{N^\sigma} \bigg) S(\{ \theta_{j}\}_j, N,n, \bm \xi) \d \bm \xi,
\end{align*}
where $S(\{ \theta_{j}\}_j, N, n,\bm \xi)$ is given by
\begin{align*}
\sum_{\substack{j,k \in \Z\\ |j| \neq |k|}} \exp \bigg\{ -\frac{1}{N^2}(j^2+k^2)+2 \pi \iu n(\theta_{j}-\theta_{k}) +2 \pi \iu \langle \bm \xi, \frac{1}{N} \begin{pmatrix} j \\ k \end{pmatrix} \rangle \bigg\},
\end{align*}
where $\langle \, \cdot \, , \, \cdot \, \rangle$ denotes the Euclidean inner product. Let us remark that the second sum is $o(1)$ as $N \to \infty$.
\end{lemma}

\begin{proof}
Let $f:\R \to \R$ and $\psi: \R^2 \to \R$ be test functions of class $\mathcal P$. The Poisson summation formula implies that $\mathrm R_2^{\sigma}(f, \psi \cdot \e^{-\| \, \cdot \, \|^2}, \{ \theta_{j} \}_j, N)$ can be rewritten as
\begin{align*}
&\widehat f(0) \frac{1}{N^2} \sum_{\substack{j,k \in \Z\\ |j| \neq |k|}} (\psi\cdot \e^{-\| \, \cdot \, \|^2}) \bigg( \frac{1}{N} \begin{pmatrix} j \\ k \end{pmatrix} \bigg) 
\\ & \qquad  +\frac{1}{N^2} \sum_{n \neq 0} \widehat f \bigg( \frac{n}{N^\sigma} \bigg) \sum_{\substack{j,k \in \Z\\ |j| \neq |k|}} \psi \bigg( \frac{1}{N} \begin{pmatrix} j \\ k \end{pmatrix} \bigg)  \e^{-\frac{1}{N^2}(j^2-k^2)}.
\end{align*}
Finally, it suffices to note that
\begin{align*}
 \psi \bigg( \frac{1}{N} \begin{pmatrix} j \\ k \end{pmatrix} \bigg) = \int_{\R^2} \widehat \psi(\bm \xi) \e \bigg(\langle \bm \xi, \frac{1}{N} \begin {pmatrix} j \\ k \end{pmatrix} \rangle \bigg) \, \d \bm \xi
\end{align*}
and switch the order of integration and summation to obtain the claim.
\end{proof}

Let us record the following observation: Suppose that any two test functions $f: \R \to \R$ and $\psi: \R^2 \to \R$ of class $\mathcal P$  (see \eqref{def:class P}) satisfy
\begin{align*}
\lim_{N \to \infty} \int_{\R^2} \widehat{\psi}(\bm \xi) \sum_{n \neq 0} \widehat f\bigg(\frac n N \bigg) S(\{ \theta_{j}\}_j, N, n ,\bm \xi) \d \bm \xi =0,
\end{align*}
where $S(\{ \theta_{j}\}_j, N, n ,\bm \xi)$ is as defined in \Cref{error term lemma}. Then, 
\begin{align*}
\lim_{N \to \infty} \mathrm R_2^{\sigma}(f, \psi \cdot \e^{-\| \, \cdot \, \|^2}, \{ \theta_{j} \}_j, N) = \widehat f(0) \mathcal F(\psi \cdot \e^{-\| \, \cdot \,\|^2})(\bm 0).
\end{align*}
Thus, it will be convenient to introduce the following error term
\begin{align*}
\mathrm E^\sigma(f, \psi, \{ \theta_{j} \}_j, N) := \frac{1}{N^2}  \int_{\R^2} \widehat{\psi}(\bm \xi) \,  \sum_{n \neq 0} \widehat f\bigg(\frac{n}{N^\sigma} \bigg) S(\{ \theta_{j}\}_j, N, n ,\bm \xi)\d \bm \xi,
\end{align*}
for any test functions $f: \R \to \R$ and $\psi: \R^2 \to \R$ of class $\mathcal P$.

\section{Passage to the Geometry of Numbers}
\label{section:passage}

\subsection{} For $\mathfrak z =\mathfrak a + \iu \mathfrak b \in \mathbb H$, where $\mathbb H$ denotes the upper half plane, and $\bm \xi \in \R^2$ let us introduce 
\begin{align*}
\Omega_{\mathfrak z, \bm \xi}(x,y) := \pi \iu\bigg(\mathfrak a(x^2-y^2) + \iu \mathfrak b(x^2+y^2) \bigg) + 2 \pi \iu \langle \bm \xi,\begin{pmatrix} x \\ y \end{pmatrix}\rangle, \, (x,y) \in \R^2
\end{align*}
as well as
\begin{align*}
\Psi_{\mathfrak z, \bm \xi}(x,y):= \exp \bigg\{\Omega_{\mathfrak z, \bm \xi}(x+y,x-y)\bigg\}, \, (x,y) \in \R^2.
\end{align*}
In this section we will show that
\begin{align*}
\sum_{\substack{(x,y) \in \Z^2 \\ |x| \neq |y|}}\e^{\Omega_{\mathfrak z, \bm \xi}(x,y)}
\end{align*}
can be thought of as the difference of two classical theta functions missing its leading terms. Many of the ideas contained in the next couple of paragraphs can be traced back to the work of G\"otze \cite{goetze:2004} on quadratic forms.

\begin{lemma}
\label{Quarter Rotation}
Let $\mathfrak z \in \mathbb H$ and $\bm \xi \in \R^2$. Then,
\begin{align*}
\sum_{\substack{(x,y) \in \Z^2 \\ |x| \neq |y|}}\e^{\Omega_{\mathfrak z, \bm \xi}(x,y)}= \sum_{\substack{(x,y) \in \Z^2 \\ xy \neq 0}}\Psi_{\mathfrak z, \bm \xi}(x,y) +  \sum_{(x,y) \in \Z^2}\Psi_{\mathfrak z, \bm \xi}\big(x-\frac{1}{2},y+\frac{1}{2}\big).
\end{align*}
\end{lemma}

\begin{proof}
For $b \in \{0,1\}$ let us set $\Z_b := \{ x \in \Z \, | \, x \equiv b \text{ mod } 2\}$. The maps
\begin{align*}
\chi^+: \Z_0^2 \cup \Z_1^2 &\longrightarrow \Z^2 \\
(x,y) & \longmapsto (\chi_1^+(x,y),\chi_2^+(x,y))=\bigg(\frac{x+y}{2}, \frac{x-y}{2} \bigg),\\
\chi^-: (\Z_0\times \Z_1) \cup (\Z_1\cup \Z_0) &\longrightarrow \Z^2 \\
(x,y) & \longmapsto (\chi_1^-(x,y),\chi_2^-(x,y))=\bigg(\frac{x+y+1}{2}, \frac{x-y-1}{2} \bigg),
\end{align*}
are bijections and with this notation note that $\big \{ (x,y) \in \Z^2 \, \big | \, |x| \neq |y| \big \}$ can be partitioned as follows
\begin{align*}
 (\Z_0 \times \Z_1)\cup (\Z_1 \times \Z_0) \cup \bigg\{ (x,y) \in \Z_0^2 \cup \Z_1^2 \, \bigg | \chi_1^+(x,y)\chi_2^+(x,y) \neq 0 \bigg\}.
 \end{align*}
 For $(x,y) \in (\Z_0 \times \Z_1)\cup (\Z_1 \times \Z_0) $ we have
\begin{align*}
\Omega_{\mathfrak z, \bm \xi}(x,y) = \Omega_{\mathfrak z, \bm \xi}(\chi_1^-(x,y)+\chi_2^-(x,y),\chi_1^-(x,y)-\chi_2^-(x,y)-1),
\end{align*}
and for $(x,y) \in \Z_0^2 \cup \Z_1^2$ we have
\begin{align*}
 \Omega_{\mathfrak z, \bm \xi}(x,y) = \Omega_{\mathfrak z, \bm \xi}(\chi_1^+(x,y)+\chi_2^+(x,y),\chi_1^+(x,y)-\chi_2^+(x,y)).
\end{align*}
\end{proof}

\begin{lemma}
\label{Cancellations Part 1}
Let $\mathfrak z \in \mathbb H$ and $\bm \xi \in \R^2$. Then,
\begin{align*}
&\sum_{\substack{(x,y) \in \Z^2 \\ |x| \neq |y|}}\e^{ \Omega_{\mathfrak z, \bm \xi}(x,y)} \\
&\qquad=\Psi_{\mathfrak z, \bm \xi}(\bm 0)+  \bigg(2 \sum_{x,y \in \Z^2} \Psi_{\mathfrak z, \bm \xi}(x,y) - \sum_{y \in \Z} \Psi_{\mathfrak z, \bm \xi}(0,y)- \sum_{x \in \Z} \Psi_{\mathfrak z, \bm \xi}(x,0) \bigg)\\
&\qquad \qquad +  \sum_{x,y \in \Z^2} \bigg(\Psi_{\mathfrak z, \bm \xi}(\frac{x}{2},\frac{y}{2})- \Psi_{\mathfrak z, \bm \xi}(x,\frac{y}{2})-\Psi_{\mathfrak z, \bm \xi}(\frac{x}{2},y) \bigg).
\end{align*}
Note here, that $\Psi_{\mathfrak z, \bm \xi}(\bm 0)=1$.
\end{lemma}

\begin{proof}
For simplicity let us write $\Omega_{\mathfrak z, \bm \xi} = \Omega$ and $\Psi_{\mathfrak z, \bm \xi}= \Psi$. According to \Cref{Quarter Rotation} we have
\begin{align*}
\sum_{\substack{(x,y) \in \Z^2 \\ |x| \neq |y|}}\exp \bigg\{ \Omega(x,y) \bigg\}= \sum_{\substack{(x,y) \in \Z^2 \\ xy \neq 0}}\Psi(x,y) +  \sum_{(x,y) \in \Z^2}\Psi\big(x-\frac{1}{2},y+\frac{1}{2}\big).
\end{align*}
Let us first prove the following claim
\begin{align*}
\sum_{(x,y) \in \Z^2}\Psi\big(x-\frac{1}{2},y+\frac{1}{2}\big) =
\sum_{\substack{(x,y) \in \Z^2\\xy \neq 0}}  \bigg(\Psi(x,y)- \Psi(\frac x 2,y)-\Psi(x,\frac y 2) + \Psi(\frac x 2,\frac{y}{2}) \bigg).
\end{align*}
Fix $x \in \Z$, then
\begin{align*}
\sum_{y \in \Z}\Psi\big(x-\frac{1}{2},y+\frac{1}{2}\big) = \sum_{y \in \Z \setminus \{0 \}} \bigg(\Psi(x-\frac{1}{2},\frac{y}{2})-\Psi(x-\frac{1}{2},y)\bigg)
\end{align*}
and summing this expression over $x \in \Z$ yields
\begin{align*}
\sum_{(x,y) \in \Z^2}\Psi\big(x-\frac{1}{2},y+\frac{1}{2}\big)  = \sum_{(x,y) \in \Z \times (\Z \setminus \{0\})} \bigg(\Psi(x-\frac{1}{2},\frac{y}{2})-\Psi(x-\frac{1}{2},y)\bigg).
\end{align*}
Finally, if we fix $y \in \Z \setminus \{0 \}$, then
\begin{align*}
&\sum_{x \in \Z} \Psi(x-\frac{1}{2},\frac{y}{2}) = \sum_{x \in \Z \setminus \{0\}} \Psi(\frac x 2,\frac y 2)-\sum_{x \in \Z \setminus \{0\}} \Psi(x,\frac y 2), \text{ and } \\
&\sum_{x \in \Z} \Psi(x-\frac{1}{2},y) =  \sum_{x \in \Z \setminus \{0\}}  \Psi(\frac x 2,y)- \sum_{x \in \Z \setminus \{0 \}}  \Psi(x,y),
\end{align*}
which proves the claim after summing these last two expressions over $y \in \Z \setminus \{0 \}$.
Hence,
\begin{align*}
\sum_{\substack{(x,y) \in \Z^2 \\ |x| \neq |y|}}\e^{ \Omega(x,y)}= \sum_{\substack{(x,y) \in \Z^2\\xy \neq 0}}  \bigg(2\Psi(x,y)- \Psi(\frac x 2,y)-\Psi(x,\frac y 2) + \Psi(\frac x 2,\frac{y}{2}) \bigg).
\end{align*}
Fix $x \in \Z \setminus \{0\}$, then
\begin{align*}
&\sum_{y \in \Z \setminus \{0 \}} \bigg(2\Psi(x,y)- \Psi(\frac x 2,y)-\Psi(x,\frac y 2) + \Psi(\frac x 2,\frac{y}{2}) \bigg) 
\\ & \qquad = \sum_{y \in \Z} \bigg(2\Psi(x,y)+ \Psi(\frac x 2,y)-\Psi(x,\frac y 2) + \Psi(\frac x 2,\frac{y}{2}) \bigg)  -\Psi(x,0)
\end{align*}
and summing this expression over $x \in \Z \setminus \{0\}$ yields
\begin{align*}
&\sum_{\substack{(x,y) \in \Z^2 \\ |x| \neq |y|}} \bigg(2\Psi(x,y)- \Psi(\frac x 2,y)-\Psi(x,\frac y 2)+ \Psi(\frac x 2,\frac{y}{2}) \bigg)  
\\ & \qquad = \sum_{(x,y) \in \Z^2} \bigg(2\Psi(x,y)-\Psi(\frac x 2,y)-\Psi(x,\frac y 2) + \Psi(\frac x 2,\frac{y}{2})    \bigg) \\ & \qquad \qquad \quad -\sum_{y \in \Z} \Psi(0,y) - \sum_{x \in \Z}\Psi(x,0) + \Psi(\bm 0).
\end{align*}
\end{proof}

\subsection{}
For any $\mathfrak z = \mathfrak a + \iu \mathfrak b \in \mathbb H$ let us define for $b \in \{0,1\}$
\begin{align}
& g_{\mathfrak z,b}:= \begin{pmatrix} \frac{1}{\sqrt{2 \mathfrak b}} & \\ & \sqrt{2 \mathfrak b} \end{pmatrix} \begin{pmatrix} 1 & 2 \mathfrak a \\ & 1 \end{pmatrix}
\begin{pmatrix} 2^b & \\ & \frac{1}{2^b} \end{pmatrix}, \\
& \Lambda_{\mathfrak z, b}:= g_{\mathfrak z ,b} \Z^2, \label{eq:Lattice correspondence}\\
& \Lambda^*_{\mathfrak z,b} = \bigg\{ \bm v = \begin{pmatrix} v_1 \\ v_2 \end{pmatrix} \in \Lambda_\mathfrak z\, \bigg | \, v_2 \neq 0 \bigg\}, \label{eq:Good lattice}
\end{align}
and in addition to this set
\begin{align}
\label{theta sum}
\mathcal C_b(\mathfrak z, \bm \xi, \bm \eta) :=\sqrt{\frac{2}{\mathfrak b}} \sum_{ \bm v\in \Lambda_{\mathfrak z,b}^*} \e^{-\pi \| \bm v + \frac{1}{\sqrt 2 \mathfrak b} \bm \xi\|^2 + 2 \pi \iu \frac{1}{\sqrt{2 \mathfrak b}} \langle \bm v, \bm \eta \rangle}
\end{align}
for any two $\bm \xi, \bm \eta \in \R^2$. Note that $\sqrt{\mathfrak b} \,\mathcal C_b(\mathfrak z, \bm \xi, \bm \eta)$ is majorized (up to a constant) from above by the Siegel transform of $\e^{-\pi \| \, \cdot \, \|^2}$ over the affine lattice $\Lambda_{\mathfrak z,b} + \frac{1}{\sqrt{2 \mathfrak b}} \bm \xi$.

\begin{lemma}
\label{Bound for Exp Sum}
Let $\mathfrak z = \mathfrak a + \iu \mathfrak b \in \mathbb H$ and $F:\R^2 \to \R$ an integrable function which is even in the second variable, that is, $F(\xi_1,-\xi_2)=F(\xi_1,\xi_2)$. Then,
\begin{align*}
&\int_{\R^2} F(\bm \xi) \bigg(\sum_{\substack{(x,y) \in \Z^2 \\ |x| \neq |y|}} \e^{\Omega_{\mathfrak z, \bm \xi}(x,y)} \bigg)\, \d \bm \xi\\ 
&\quad= \int_{\R^2} F(\bm \xi) \bigg( \Psi_{\mathfrak z, \bm \xi}(\bm 0) - \sqrt{\frac{2}{\mathfrak b}} \sum_{x \in \Z} \e^{-\frac{\pi}{2 \mathfrak b}(2 x +1 +\langle \bm \xi, \bar{\bm 1} \rangle)^2}\bigg) \, \d \bm \xi \\
& \qquad + \int_{\R^2} F(\bm \xi) \bigg( \mathcal C_0(\mathfrak z,\langle \bm \xi, \bar{\bm 1} \rangle \bm e_1,\langle \bm \xi, \underline{\bm 1} \rangle \bm e_2) - \mathcal C_1(\mathfrak z,\langle \bm \xi, \bar{\bm 1} \rangle \bm e_1,\langle \bm \xi, \underline{\bm 1} \rangle \bm e_2)\bigg) \, \d \bm \xi,
\end{align*} 
where $\bar{\bm 1}= \begin{pmatrix} 1 \\ 1 \end{pmatrix}$, $\underline{\bm 1} = \begin{pmatrix} 1 \\ -1 \end{pmatrix}$ and $\mathcal C_b(\mathfrak z, \bm \xi, \bm \eta)$ is defined as in \eqref{theta sum}.
\end{lemma}

\begin{proof}
Let $\mathfrak z = \mathfrak a + \iu \mathfrak b \in \mathbb H$ and $\xi \in \R$. Recall from \Cref{Cancellations Part 1} that
\begin{align*}
&\sum_{\substack{(x,y) \in \Z^2 \\ |x| \neq |y|}}\e^{ \Omega_{\mathfrak z, \bm \xi}(x,y)} \\
&\qquad=\Psi(\bm 0)+  \bigg(2 \sum_{x,y \in \Z^2} \Psi_{\mathfrak z, \bm \xi}(x,y) - \sum_{y \in \Z} \Psi_{\mathfrak z, \bm \xi}(0,y)- \sum_{x \in \Z} \Psi_{\mathfrak z, \bm \xi}(x,0)\bigg) \\
&\qquad \qquad +  \sum_{x,y \in \Z^2} \bigg(\Psi_{\mathfrak z, \bm \xi}(\frac{x}{2},\frac{y}{2})- \Psi_{\mathfrak z, \bm \xi}(x,\frac{y}{2})-\Psi_{\mathfrak z, \bm \xi}(\frac{x}{2},y) \bigg).
\end{align*}
 A straightforward calculation shows that
\begin{align*}
\Omega_{\mathfrak z, \bm \xi}(x+y,x-y) = \pi \iu \bigg( \mathfrak a 4 x y + \iu 2 \mathfrak b(x^2+y^2) \bigg) + 2\pi \iu x \langle \bm \xi, \bar{\bm 1} \rangle +2 \pi \iu y \langle \bm \xi, \underline{\bm 1} \rangle,
\end{align*}
where
$\bar{\bm 1}= \begin{pmatrix} 1 \\ 1 \end{pmatrix}$ and $\underline{\bm 1} = \begin{pmatrix} 1 \\ -1 \end{pmatrix}$. For any $r,s >0$ let us state the following identities 
\begin{align*}
&\sum_{x,y \in \Z} \Psi_{\mathfrak z, \bm \xi}(rx,sy) = \frac{1}{r\sqrt{2 \mathfrak b}}\sum_{x,y\in \Z} \e^{-\pi s^2 2 \mathfrak b y^2 - \frac{\pi}{r^22 \mathfrak b}(x+2 rs \mathfrak a  y + r \langle \bm \xi, \bar{\bm 1} \rangle )^2+2 \pi \iu y \langle s \bm \xi, \underline{\bm 1} \rangle},\\
&\sum_{x,y \in \Z} \Psi_{\mathfrak z, \bm \xi}(rx,sy) = \frac{1}{s\sqrt{2 \mathfrak b}}\sum_{x,y\in \Z} \e^{-\pi r^2 2 \mathfrak b x^2 - \frac{\pi}{s^2 2 \mathfrak b}(y+2 rs \mathfrak a  x + s \langle \bm \xi, \underline{\bm 1} \rangle )^2+2 \pi \iu x \langle r \bm \xi, \bar{\bm 1} \rangle}, 
\\ & \sum_{x \in \Z} \Psi_{\mathfrak z, \bm \xi}(rx,0) = \frac{1}{r\sqrt{2 \mathfrak b}}\sum_{x\in \Z} \e^{- \frac{\pi}{r^22 \mathfrak b}(x + r \langle \bm \xi, \bar{\bm 1} \rangle )^2},
\\ & \sum_{y \in \Z} \Psi_{\mathfrak z, \bm \xi}(0,sy) = \frac{1}{s\sqrt{2 \mathfrak b}}\sum_{y\in \Z} \e^{ - \frac{\pi}{s^2 2 \mathfrak b}(y + s \langle \bm \xi, \underline{\bm 1} \rangle )^2}.
\end{align*}
Let us only prove the first identity as the second one is analogous to the first one and the latter two can be inferred from the former two. Indeed, note that
\begin{align*}
\sum_{x,y \in \Z} \Psi_{\mathfrak z, \bm \xi}(rx,sy) &= \sum_{y \in \Z} \e^{-2 \pi s^2\mathfrak b y^2+2 \pi \iu y \langle s \bm \xi, \underline{\bm 1}\rangle} \bigg( \sum_{x \in \Z} \e^{\pi \iu(\iu 2 r^2 \mathfrak b x^2)+ 2 \pi \iu x(2 r s \mathfrak a y + \langle r \bm \xi, \bar{\bm 1} \rangle)}\bigg) 
\\ &= \frac{1}{r\sqrt{2 \mathfrak b}} \sum_{x,y \in \Z} \e^{-2 \pi  s^2\mathfrak b y^2+2 \pi \iu y \langle s \bm \xi, \underline{\bm 1}\rangle - \frac{\pi}{r^2 2 \mathfrak b}(x +2 r s  \mathfrak a y + \langle r \bm \xi, \bar{\bm 1}\rangle)^2},
\end{align*}
where we applied Poisson's summation formula to the inner sum over $x$.

We first deduce, using the above identities, that
\begin{align*}
&2\sum_{x,y \in \Z} \Psi_{\mathfrak z, \bm \xi}(x,y) -\sum_{x \in \Z} \Psi_{\mathfrak z, \bm \xi}(x,0)-\sum_{y \in \Z} \Psi_{\mathfrak z, \bm \xi}(0,y)
\\  & \qquad= \frac{1}{\sqrt{2 \mathfrak b}}\ \sum_{(x,y)\in \Z \times (\Z \setminus \{0\})} \e^{-\pi 2 \mathfrak b y^2 - \frac{\pi}{2 \mathfrak b}(x+2 \mathfrak a  y + \langle \bm \xi, \bar{\bm 1} \rangle )^2+2 \pi \iu y \langle  \bm \xi, \underline{\bm 1} \rangle} \\ & \qquad \qquad +  \frac{1}{\sqrt{2 \mathfrak b}}\sum_{(x,y)\in (\Z \setminus \{0\}) \times \Z }  \e^{-\pi 2 \mathfrak b x^2 - \frac{\pi}{2 \mathfrak b}(y+2 \mathfrak a  x +  \langle \bm \xi, \underline{\bm 1} \rangle )^2+2 \pi \iu x \langle  \bm \xi, \bar{\bm 1} \rangle}.
\end{align*}
Hence, if $F$ is an even function in the second variable, it is plain to see that
\begin{align*}
&\int_{\R^2}F(\bm \xi)\bigg(2\sum_{x,y \in \Z} \Psi_{\mathfrak z, \bm \xi}(x,y) -\sum_{x \in \Z} \Psi_{\mathfrak z, \bm \xi}(x,0)-\sum_{y \in \Z} \Psi_{\mathfrak z, \bm \xi}(0,y) \bigg) \, \d \bm \xi \\
& \quad =\sqrt{\frac{2}{\mathfrak b}} \int_{\R^2} \widehat \psi(\bm \xi) \sum_{(x,y)\in \Z \times (\Z \setminus \{0\})} \e^{-\pi 2 \mathfrak b y^2 - \frac{\pi}{2 \mathfrak b}(x+2 \mathfrak a  y + \langle \bm \xi, \bar{\bm 1} \rangle )^2 + 2 \pi \iu y \langle  \bm \xi, \underline{\bm 1} \rangle} \, \d \bm \xi.
\end{align*}

On the other hand, we have
\begin{align*}
&\frac{1}{2} \sum_{x,y \in \Z} \Psi_{\mathfrak z, \bm \xi}(\frac x 2,\frac y 2) \\ 
& \quad= \frac{1}{\sqrt{2 \mathfrak b}}\sum_{(x,y)\in \Z \times \Z} \e^{-\pi 2 \mathfrak b (\frac{y}{2})^2 - \frac{\pi}{\mathfrak 2b}(2x+ 2 \mathfrak a \frac{y}{2} +  \langle \bm  \xi, \bar{\bm 1} \rangle )^2+2 \pi \iu  \frac{y}{2}  \langle \bm \xi, \underline{\bm 1} \rangle} 
\end{align*}
and
\begin{align*}
& \sum_{x,y \in \Z} \Psi_{\mathfrak z, \bm \xi}(x,\frac y 2) \\ 
& \quad = \frac{1}{\sqrt{2 \mathfrak b}}\sum_{(x,y)\in \Z \times \Z } \e^{-\pi 2\mathfrak b (\frac{y}{2})^2 - \frac{\pi}{2 \mathfrak b}(x+ 2 \mathfrak a  \frac{y}{2} +  \langle \bm \xi, \bar{\bm 1} \rangle )^2+2 \pi \iu \frac{y}{2} \langle  \bm \xi, \underline{\bm 1} \rangle} \\ 
& \quad = \frac{1}{2} \sum_{x,y \in \Z} \Psi_{\mathfrak z, \bm \xi}(\frac x 2,\frac y 2)  +\frac{1}{\sqrt{2 \mathfrak b}} \sum_{(x,y)\in \Z \times \Z } \e^{-\pi 2\mathfrak b (\frac{y}{2})^2 - \frac{\pi}{2 \mathfrak b}(2x+1+ 2 \mathfrak a  \frac{y}{2} +  \langle \bm \xi, \bar{\bm 1} \rangle )^2+2 \pi \iu \frac{y}{2} \langle  \bm \xi, \underline{\bm 1} \rangle}.
\end{align*}
Similarly, by symmetry from the previous case, we deduce
\begin{align*}
&\frac{1}{2}\sum_{x,y \in \Z} \Psi_{\mathfrak z, \bm \xi}(\frac x 2,\frac y 2) \\
 & \quad = \frac{1}{\sqrt{2 \mathfrak b}}\sum_{(x,y)\in \Z \times \Z} \e^{-\pi  2 \mathfrak b (\frac{x}{2})^2 - \frac{\pi}{2\mathfrak b}(2y +2 \mathfrak a \frac{x}{2}   +  \langle \bm \xi, \underline{\bm 1} \rangle )^2+2 \pi \iu \frac{x}{2} \langle \bm \xi, \bar{\bm 1} \rangle}
\end{align*}
and
\begin{align*}
& \sum_{x,y \in \Z} \Psi_{\mathfrak z, \bm \xi}(\frac x 2,y) \\ & \quad = \frac{1}{2}\sum_{x,y \in \Z} \Psi_{\mathfrak z, \bm \xi}(\frac x 2,\frac y 2)+ \frac{1}{\sqrt{2 \mathfrak b}}\sum_{(x,y)\in \Z\times \Z} \e^{-\pi 2\mathfrak b(\frac{x}{2})^2 - \frac{\pi}{2 \mathfrak b}(2y+1+ 2\mathfrak a \frac{x}{2} +  \langle \bm \xi, \underline{\bm 1} \rangle )^2+2 \pi \iu \frac{x}{2} \langle  \bm \xi, \bar{\bm 1} \rangle}.
\end{align*}
Thus, we conclude that
\begin{align*}
&\sum_{x,y \in \Z} \bigg(\Psi_{\mathfrak z, \bm \xi}(\frac x 2,\frac y 2) - \Psi_{\mathfrak z, \bm \xi}(x ,\frac y 2) - \Psi_{\mathfrak z, \bm \xi}(\frac x 2 ,y) \bigg) \\
& \quad= -\frac{1}{\sqrt{2 \mathfrak b}} \sum_{(x,y)\in \Z \times \Z } \e^{-\pi 2\mathfrak b (\frac{y}{2})^2 - \frac{\pi}{2 \mathfrak b}(2x+1+ 2 \mathfrak a  \frac{y}{2} +  \langle \bm \xi, \bar{\bm 1} \rangle )^2+2 \pi \iu \frac{y}{2} \langle  \bm \xi, \underline{\bm 1} \rangle}
\\ & \quad \quad \quad - \frac{1}{\sqrt{2 \mathfrak b}}\sum_{(x,y)\in \Z\times \Z} \e^{-\pi 2\mathfrak b(\frac{x}{2})^2 - \frac{\pi}{2 \mathfrak b}(2y+1+ 2\mathfrak a \frac{x}{2} +  \langle \bm \xi, \underline{\bm 1} \rangle )^2+2 \pi \iu \frac{x}{2} \langle  \bm \xi, \bar{\bm 1} \rangle}.
\end{align*}
and after integrating this expression against $F$, we obtain
\begin{align*}
&\int_{\R^2} F(\bm \xi) \sum_{x,y \in \Z} \bigg(\Psi_{\mathfrak z, \bm \xi}(\frac x 2,\frac y 2) - \Psi_{\mathfrak z, \bm \xi}(x ,\frac y 2) - \Psi_{\mathfrak z, \bm \xi}(\frac x 2 ,y) \bigg) \, \d \bm \xi \\
& \quad = - \sqrt{\frac{2}{\mathfrak b}} \int_{\R^2} F(\bm \xi)\sum_{(x,y)\in \Z \times \Z} \e^{-\pi 2\mathfrak b (\frac{y}{2})^2 - \frac{\pi}{2 \mathfrak b}(2x+1+ 2 \mathfrak a  \frac{y}{2} +  \langle \bm \xi, \bar{\bm 1} \rangle )^2+2 \pi \iu \frac{y}{2} \langle  \bm \xi, \underline{\bm 1} \rangle} \, \d \bm \xi.
\end{align*}
It is then easy to see that
\begin{align*}
&\int_{\R^2}F(\bm \xi) \sum_{\substack{(x,y) \in \Z^2 \\ |x| \neq |y|}}\e^{\Omega_{\mathfrak z, \bm \xi}(x,y)} \, \d \bm \xi 
\\& \quad = \int_{\R^2}F(\bm \xi)\bigg( \Psi_{\mathfrak z, \bm \xi}(\bm 0) + \sqrt{\frac{2}{\mathfrak b}} \sum_{(x,y)\in \Z \times (\Z \setminus \{0\})} \e^{-\pi 2 \mathfrak b y^2 - \frac{\pi}{2 \mathfrak b}(x+2 \mathfrak a  y + \langle \bm \xi, \bar{\bm 1} \rangle )^2 + 2 \pi \iu y \langle  \bm \xi, \underline{\bm 1} \rangle}
\\ & \quad \quad   - \sqrt{\frac{2}{\mathfrak b}}\sum_{(x,y)\in \Z \times \Z} \e^{-\pi 2\mathfrak b (\frac{y}{2})^2 - \frac{\pi}{2 \mathfrak b}(2x+1+ 2 \mathfrak a  \frac{y}{2} +  \langle \bm \xi, \bar{\bm 1} \rangle )^2+2 \pi \iu \frac{y}{2} \langle  \bm \xi, \underline{\bm 1} \rangle} \bigg)\, \d \bm \xi
\end{align*}
\end{proof}

\begin{lemma}
\label{Reduction to Geometry of Numbers}
Let $\theta_{j} = \displaystyle{\alpha j^2}$ and let $f: \R \to \R$ and $\psi: \R^2 \to \R$ be test functions of class $\mathcal P$. Then, for any $\epsilon >0$ there is $N_0 = N_0(\epsilon)$ such that for any $N \geq N_0$
\begin{align*}
\mathrm E^\sigma(f,\psi,\{\theta_{j}\}_j, N) \ll_{\psi, f} \frac{1}{N} + \mathrm E^{\sigma}_\epsilon(\psi,\{\theta_{j}\}_j, N),
\end{align*}
where $\mathrm E^{\sigma}_\epsilon(\psi,\{\theta_{j}\}_j, N)$ is defined as
\begin{align*}
 \frac{1}{N^2} \int_{\R^2} |\widehat \psi(\bm \xi) | \sum_{1 \leq |n| \leq N^{\sigma + \epsilon}} \sum_{b \in \{0,1\}} \mathcal C_b\bigg(\mathfrak z_{n,N}^{\alpha},\sqrt{\frac{\pi}{2}}N \langle \bm \xi, \bar{\bm 1} \rangle \bm e_1,\bm 0\bigg) \, \d \bm \xi,
\end{align*}
and 
\begin{align*}
\mathfrak z_{n,N}^{\alpha} =2 n \alpha + \iu \frac{1}{\pi N^2} \in \mathbb H.
\end{align*}
\end{lemma}

\begin{proof}
Recall that $\mathrm E^\sigma(f,\psi,\{ \theta_{j} \}_j, N)$ is given by
\begin{align*}
\bigg |\frac{1}{N^2} \int_{\R^2} \widehat{\psi}(\bm \xi) \, \sum_{n \neq 0} \widehat f\bigg(\frac{n}{N^\sigma} \bigg) S(\{ \theta_j\}_j, N, n ,\bm \xi)\d \bm \xi, \bigg|
\end{align*}
where $S(\{ \theta_j\}_j, N, n ,\bm \xi)$ is defined as in \Cref{error term lemma}. It is easy to see that 
\begin{align*}
S(\{ \theta_{j}\}_j, N, n ,\bm \xi)  = \sum_{\substack{(x,y) \in \Z^2 \\ |x| \neq |y|}} \exp \bigg\{ \Omega_{\mathfrak z_{n,N}^{\alpha}, \frac{\bm \xi}{N}}(x,y) \bigg\},
\end{align*}
where
\begin{align*}
\mathfrak z_{n,N}^{\alpha} = 2 n \alpha + \iu \frac{1}{\pi N^2} \in \mathbb H.
\end{align*}
Since, $\widehat{\psi}$ has compact support, we can choose $N_0$ such that for all $N \geq N_0$, $\| \frac{\bm \xi}{N} \| \leq \frac{1}{2}$ for all $\bm \xi \in \text{supp}(\widehat \psi)$ and in addition
\begin{align*}
2 \pi N  \sum_{x \in \Z} \e^{-\pi^2 N(2 x +1 +\langle \bm \xi, \bar{\bm 1} \rangle)^2} \leq 1.
\end{align*}
Thus, according to \Cref{Bound for Exp Sum}, $\mathrm E^\sigma(f,\psi,\{ \theta_{j} \}_j, N) $ is bounded from above, up to a constant depending on $\psi$, by
\begin{align*}
\frac{1}{N} + \frac{1}{N^2} \int_{\R^2} |\widehat \psi(\bm \xi) | \sum_{n \in \Z \setminus \{0\}} \bigg| f\bigg( \frac{n}{N^\sigma}\bigg) \bigg| \sum_{b \in \{0,1\}} \mathcal C_b\bigg(\mathfrak z_{n,N}^{\alpha},\sqrt{\frac{\pi}{2}}N \langle \bm \xi, \bar{\bm 1} \rangle \bm e_1,\bm 0\bigg) \, \d \bm \xi, 
\end{align*}
for all $N \geq N_0$. 
Let us now eliminate the tail of the sum over $n$ for the right-hand side expression in the former equation. Since $\widehat f$ is compactly supported or equal to $\e^{-2 \pi | \, \cdot  \,|}$, for any $\epsilon >0$, we can assume that
\begin{align*}
\bigg | \widehat f \bigg( \frac{n}{N^\sigma} \bigg) \bigg| \leq \e^{-N^\epsilon-\frac{|n|}{N^\sigma}} , \text{ for all } |n| \geq N^{\sigma+\epsilon}.
\end{align*} 
Thus, a geometric series argument together with the trivial upper bound $|\mathcal C(\mathfrak z_{n,N}^{\alpha,\sigma},\bm \xi, \bm \eta)| \ll N^2$, implies that
\begin{align*}
 \frac{1}{N^2} \int_{\R^2} |\widehat \psi(\bm \xi) |  \sum_{|n| \geq N^{\sigma+\epsilon}} \bigg|\widehat f\bigg(\frac{n}{N^\sigma} \bigg) \bigg| \sum_{u \in U_{\frac{\bm \xi}{N}}} \mathcal C(\mathfrak z_{n,N}^{\alpha,\sigma},u) \ll_{\psi} N^\sigma \e^{-N^\epsilon}.
\end{align*}
Thus, we can enlarge $N_0$ (depending on $\epsilon$) once again, to require
\begin{align*}
 \frac{1}{N^2} \int_{\R^2} |\widehat \psi(\bm \xi) |  \sum_{|n| \geq N^{\sigma+\epsilon}} \bigg|\widehat f\bigg(\frac n N \bigg) \bigg| \sum_{b \in \{0,1\}} \bigg|\mathcal C_b\bigg(\mathfrak z_{n,N}^{\alpha},\sqrt{\frac{\pi}{2}}N \langle \bm \xi, \bar{\bm 1} \rangle \bm e_1,\bm 0\bigg) \bigg| \, \d \bm \xi\leq \frac{1}{N},
\end{align*}
for all $N \geq N_0$.
\end{proof}

\section{An Argument from the Geometry of Numbers}

\subsection{} \label{subsection:lattice point count} For a two-dimensional unimodular lattice $\Delta \subset \R^2$ denote by $a(\Delta)$ the first successive minimum (its reciprocal is commonly referred to as the height) of $\Delta$ (i.e. the length of the shortest non-zero vector in $\Delta$); It is well-known that $a(\Delta)\leq \frac{2}{\sqrt 3}$. For any lattice $\Delta \subset \R^2$ and any $\mu>0$ let us denote by
\begin{align}
\mathcal N_{\Delta}(\mu) := \bigg \{ \bm v \in \Delta \, \bigg | \, \| \bm v \| \leq \mu \bigg\},
\end{align}
the number of lattice points contained in the disk of radius $\mu$. The following three Lemmas are basic results in the Geometry of Numbers. They can all be generalized to higher dimensions, but we restrict it here to the case of two dimensional unimodular lattices.

\subsection{} The Lipschitz principle in the Geometry of Numbers \cite{davenport:1951} is a well-known counting method for the number of lattice points in a bounded region in terms of the successive minima:
\begin{lemma}
\label{unimodular Lipschitz principle}
Let $\Delta$ be a two-dimensional unimodular lattice and $\mu>0$. Then,
\begin{align*}
\mathcal N_\Delta(\mu) \ll \begin{cases} 1+\mu a(\Delta)^{-1}, & \text{ if } a(\Delta)\leq \mu < a(\Delta)^{-1} \\ 1 +\mu^2, & \text{ if } \mu \geq a(\Delta)^{-1}. \end{cases}
\end{align*}
The implicit constant is independent of $\Delta$.
\end{lemma}

\begin{proof}
Let us write $\Delta =g \Z^2$ with $g \in \SL_2(\R)$. It is well-known that there are $\gamma \in \SL_2(\Z), k \in \SO(2)$ and $s \in [-\frac{1}{2},\frac{1}{2}]$ such that
\begin{align*}
g \gamma = k \begin{pmatrix} a(\Delta) & \\ & a(\Delta)^{-1} \end{pmatrix}  \begin{pmatrix} 1 & s \\ & 1 \end{pmatrix}.
\end{align*}
Thus, if 
\begin{align*}
\bm v =k \begin{pmatrix} a(\Delta) & \\ & a(\Delta)^{-1} \end{pmatrix}  \begin{pmatrix} 1 & s \\ & 1 \end{pmatrix} \begin{pmatrix} x \\ y \end{pmatrix} \in \Delta, \; x,y \in \Z
\end{align*} 
satisfies $\| \bm v \| \leq \mu$, then
\begin{align*}
&a(\Delta) | x + s y | \leq \mu, \\
&a(\Delta)^{-1} |y| \leq \mu.
\end{align*}
If $a(\Delta)\leq \mu < a(\Delta)^{-1}$, then $y = 0$ and thus $|x| \leq a(\Delta)^{-1} \mu$, which proves the first assertion. If $\mu \geq a(\Delta)^{-1}$, then there are at most $1 +2 a(\Delta) \mu$ possibilities for $y$;  And for each such $y$ any $x \in \Z$ satisfying $a(\Delta) | x + s y | \leq \mu$ lies in an interval of length $2 a(\Delta)^{-1}\mu$, from which we deduce that there are at most $1+2 a(\Delta)^{-1}\mu$ such integers.  Thus, we conclude that there are at most $(1 +2 a(\Delta)\mu)(1+2 a(\Delta)^{-1}\mu)  \leq 1 +\frac{4}{\sqrt 3} \mu+ 6 \mu^2$ possible lattice points $\bm v \in \Delta$ satisfying $\| \bm v \| \leq \mu$ in this case.
\end{proof}

\subsection{} \label{subsection:diophantine lattices}
In the following it will be convenient to introduce the following notation: For any $P, Q >0$ and $\alpha \in \R$ let us introduce the lattice
\begin{align*}
\Delta_{P,Q}^{\alpha} := \begin{pmatrix} P & \\ & Q \end{pmatrix} \begin{pmatrix} 1 & \alpha \\ & 1 \end{pmatrix} \Z^2.
\end{align*}
Moreover, we shall simply write $\Delta_{P} ^{\alpha}$ to denote $\Delta_{P,P^{-1}} ^{\alpha}$. i.e.
\begin{align*}
\Delta_{P}^{\alpha} := \begin{pmatrix} P & \\ & P^{-1} \end{pmatrix} \begin{pmatrix} 1 & \alpha \\ & 1 \end{pmatrix} \Z^2.
\end{align*}
The height of such lattices is known to exhibit a better upper bound in terms of $P$ than the trivial one when $\alpha$ is Diophantine.

\begin{lemma}
\label{Height estimate}
Let $\alpha \in \R$ be Diophantine of type $\kappa$. 
Then, for any $P > \frac{2}{\sqrt 3}$, we have
\begin{align*}
a(\Delta_{P}^{\alpha})^{-1} \leq c^{-\frac{1}{\kappa}} \, P^{1-\frac{2}{\kappa}}.
\end{align*}
\end{lemma}

\begin{proof}
Let $(p,q)=1$ be a coprime integer pair such that
\begin{align*}
a(\Delta_{P}^{\alpha})^{2} =  P^2(p+\alpha q)^2+P^{-2} q^2,
\end{align*}
and note that $q \neq 0$ as long as $P > \frac{2}{\sqrt 3}$. Thus,
\begin{align*}
a(\Delta_{P}^{\alpha}) \geq P|p+\alpha q| \geq \frac{c\, P}{q^{\kappa-1}} \geq \frac{c \, a(\Delta_{P}^{\alpha})^{-\kappa+1}}{P^{\kappa-2}},
\end{align*}
which proves the assertion.
\end{proof}

\begin{lemma}
\label{Siegel transform estimate}
Let $P, Q, \alpha \in \R, C>0$ and $\bm u \in \R^2$. Then, for any $Z, \zeta >0$ 
\begin{align*}
\sum_{\bm v \in \Delta_{P,Q}^\alpha} \e^{-C\|\bm v+\bm u\|^2} \leq \bigg(1+ \mathcal N_{\Delta_{(Z \zeta^{-1})^{1/2}}^\alpha}(\sqrt 2(Z \zeta)^\frac{1}{2}) \bigg) \sum_{\bm v \in \Delta^0 _{P\zeta, Q Z}} \e^{-C\frac{\| \bm v\|^2}{4}}.
\end{align*}
In particular, if $P \zeta \gg 1$ and $Q Z \gg 1$, then
\begin{align*}
\sum_{\bm v \in \Delta_{P,Q}^\alpha} \e^{-C\|\bm v+\bm u\|^2} \ll 1+ \mathcal N_{\Delta_{(Z \zeta^{-1})^{1/2}}^\alpha}(\sqrt 2(Z \zeta)^\frac{1}{2}).
\end{align*}
\end{lemma}

\begin{proof}
For simplicity we prove this for $C=1$. Any element $\bm v \in \Delta_{P,Q} ^\alpha$ is of the form
\begin{align*}
\bm v = \begin{pmatrix} P (x +\alpha z) \\ Q z \end{pmatrix}, \; x, z \in \Z
\end{align*}
and so
\begin{align*}
\| \bm v + \bm u \|^2 &= P^2 \bigg( x + \alpha z + \frac{u_1}{P} \bigg)^2 + Q^2 \bigg( z + \frac{u_2}{Q} \bigg)^2 \\ &= (P \zeta)^2 \frac{1}{\zeta^2} \bigg( x + \alpha z + \frac{u_1}{P} \bigg)^2 + (Q Z)^2 \frac{1}{Z^2}\bigg( z + \frac{u_2}{Q} \bigg)^2.
\end{align*}
For any $x, z\in \Z$ (and consequently $\bm v$) there is a unique $\bm m =(m_1,m_2) \in \Z^2$ such that
\begin{align}
\begin{aligned}
\label{eq:change of basis}
&\frac{1}{\zeta} \bigg( x + \alpha z + \frac{u_1}{P} \bigg) \in \bigg [m_1- \frac{1}{2}, m_1+\frac{1}{2}\bigg ), \\
& \frac{1}{Z} \bigg( z + \frac{u_2}{Q} \bigg)\in \bigg [m_2- \frac{1}{2}, m_2+\frac{1}{2}\bigg ).
\end{aligned}
\end{align}
Observe, in this case, that
\begin{align*}
&\bigg | \frac{1}{\zeta} \bigg( x + \alpha z + \frac{u_1}{P} \bigg) \bigg| \geq \frac{|m_1|}{2}, \\
& \bigg | \frac{1}{Z} \bigg( z + \frac{u_2}{Q} \bigg) \bigg| \geq \frac{|m_2|}{2}.
\end{align*}
Let us denote by $\mathcal B_{\bm m}(\zeta,Z)$ the collection of $\bm v \in \Delta$ satisfying \eqref{eq:change of basis}, then 
\begin{align*}
\sum_{\bm v \in \Delta} \e^{-\| \bm v + \bm u\|^2} &\leq \sum_{\bm m \in \Z^2} \sum_{\bm v \in \mathcal B_{\bm m}(\zeta,Z)} \e^{-\| \bm v + \bm u \|} \\ &\leq \sum_{ \bm m \in \Z^2} \# \mathcal B_{\bm m}(\zeta,Z)\, \e^{-\frac{1}{4}\big( (P\zeta)^2 m_1^2 + (QZ)^2 m_2^2\big)}.
\end{align*}
For each $\bm m \in \Z^2$, either $\# \mathcal B_{\bm m}(\zeta,Z) \leq 1$ or $ \mathcal B_{\bm m}(\zeta,Z)$ contains at least two elements. In the latter case, fix an element of the form 
\begin{align*}
\bm w_0 = \begin{pmatrix}  \frac{1}{\zeta}(x_0+ \alpha z_0 + \frac{u_1}{P}) \\ \frac{1}{Z} (z_0 + \frac{u_2}{Q} )   \end{pmatrix}, \, x_0, z_0 \in \Z 
\end{align*}
satisfying \eqref{eq:change of basis}, then for any other
\begin{align*}
\bm w = \begin{pmatrix}  \frac{1}{\zeta}(x+ \alpha z + \frac{u_1}{P}) \\ \frac{1}{Z} (z + \frac{u_2}{Q} )   \end{pmatrix}, \, x, z \in \Z,
\end{align*}
satisyfing \eqref{eq:change of basis} note that
\begin{align*}
\bm w - \bm w_0 =  \begin{pmatrix}  \frac{1}{\zeta}((x-x_0)+ \alpha (z-z_0)) \\ \frac{1}{Z} (z-z_0 )  \end{pmatrix} \in [-1,1]^2.
\end{align*}
Thus, if $\mathcal B_{\bm m}(\zeta,Z)$ contains at least two elements, then
\begin{align*}
\# \mathcal B_{\bm m}(\zeta,Z) \leq \# \bigg \{ \begin{pmatrix} x \\ z \end{pmatrix} \in \Z^2 \, \bigg | \, |x +\alpha z| \leq \zeta, \, |z | \leq Z \, \bigg \}.
\end{align*}
Finally observe that the inequalities
\begin{align*} 
|x +\alpha z | \leq \zeta, \; |z | \leq Z
\end{align*}
can be rewritten as
\begin{align*}
 \bigg( \frac{Z}{\zeta} \bigg)^\frac{1}{2} |x +\alpha z | \leq (Z\zeta)^\frac{1}{2} ,\; \bigg( \frac{\zeta}{Z} \bigg)^\frac{1}{2} |z | \leq (Z\zeta)^\frac{1}{2} 
\end{align*}
and hence
\begin{align*}
& \# \bigg \{ \begin{pmatrix} x \\ z \end{pmatrix} \in \Z^2 \, \bigg | \, |x +\alpha z| \leq \zeta, \, |z | \leq Z \, \bigg \}
\\ & \qquad = \# \bigg \{ \bm v \in \Delta^\alpha _{(Z \zeta^{-1})^{1/2}} \, \bigg | \, \bm v \in [- (Z\zeta)^\frac{1}{2}, (Z\zeta)^\frac{1}{2}]^2 \bigg \} \\ & \qquad \leq \mathcal N_{\Delta^\alpha _{(Z \zeta^{-1})^{1/2}}}(\sqrt 2 (Z\zeta)^\frac{1}{2}).
\end{align*}
\end{proof}

\subsection{} The following lemma is the main step towards \Cref{Theorem 1}.
\begin{lemma}
\label{First estimate}
Let $\sigma>0, b \in \{0,1\}$ and let $\alpha \in \R$ be Diophantine of type $\kappa$. Then, for any $\epsilon>0$ and any $u \in \R$
\begin{align*}
\frac{1}{N} \sum_{1 \leq n \leq N^{\sigma+\epsilon}} \sum_{\bm v \in \Lambda_{\mathfrak z_{n,N}^{\alpha},b}^*} \e^{-\pi \| \bm v+u \bm{e_1} \|^2} \ll_\epsilon N^{\sigma-1+\epsilon} + N^{\sigma-\frac{2+\sigma}{\kappa}+\epsilon(1-\frac{1}{\kappa})},
\end{align*}
where $\mathfrak z_{n,N}^{\alpha}$ is defined in \Cref{Reduction to Geometry of Numbers} and $\Lambda_{\mathfrak z_{n,N}^{\alpha},b} \subset \Lambda_{\mathfrak z_{n,N}^{\alpha},b}^*$ is defined as in \eqref{eq:Lattice correspondence} and \eqref{eq:Good lattice}. Note that the right-hand side decays to zero as $N \to \infty$ as long as $\sigma<1$ and $\kappa <1+ \frac{2}{\sigma}$.
\end{lemma}

\begin{proof}
Let us prove this for $b=0$ only, as the case $b=1$ is exactly the same. Let $\epsilon>0$ to be determined below. Denote by $\tau(\, \cdot \, )$ the divisor function and observe that
\begin{align*}
&\frac{1}{N} \sum_{1 \leq n \leq N^{\sigma+\epsilon}} \sum_{\bm v \in \Lambda_{\mathfrak z_{n,N}^{\alpha}}^*} \e^{-\pi \| \bm v+u \bm{e_1}\|^2} 
\\ & \quad= \frac{1}{N} \sum_{1 \leq n \leq N^{\sigma+\epsilon}} \sum_{(x,y) \in \Z \times (\Z \setminus \{0 \})} \e^{- 2 \pi^2 N^2(x+ 4\alpha n y + \frac{u}{\sqrt{2 \pi}N})^2- 2\frac{y^2}{N^2}}
\\ & \quad \leq \frac{1}{N} \sum_{1 \leq n \leq N^{\sigma+\epsilon}} \sum_{(x,y) \in \Z \times (\Z \setminus \{0 \})} \e^{- 2 \pi^2 N^2(x+ 4\alpha n y + \frac{u}{\sqrt{2 \pi}N})^2- \frac{1}{8} \frac{(4yn)^2}{N^{2(1+\sigma+ \epsilon)}}}
\\ & \quad \leq \frac{2}{N} \sum_{z=1} ^\infty \tau(z) \sum_{ x \in \Z} \e^{- 2 \pi^2 N^2(x+ \alpha z + \frac{u}{\sqrt{2 \pi}N})^2- \frac{1}{8} \frac{z^2}{N^{2(1+\sigma+ \epsilon)}}}.
\end{align*}
Using the fact that $\tau(z) \ll_{\epsilon,\sigma} z^{\frac{\epsilon}{2(1+\sigma+2\epsilon)}}$, we first deduce
\begin{align*}
&\sum_{1 \leq z \leq N^{2(1+\sigma+2\epsilon)}} \tau(z) \sum_{ x \in \Z} \e^{- 2 \pi^2 N^2(x+ \alpha z + \frac{u}{\sqrt{2 \pi}N})^2- \frac{1}{8} \frac{z^2}{N^{2(1+\sigma+ \epsilon)}}} 
\\ &\qquad \ll_{\epsilon,\sigma} N^{\epsilon} \sum_{1 \leq z \leq N^{2(1+\sigma+2\epsilon)}} \sum_{ x \in \Z} \e^{- 2 \pi^2 N^2(x+ \alpha z + \frac{u}{\sqrt{2 \pi}N})^2- \frac{1}{8} \frac{z^2}{N^{2(1+\sigma+ \epsilon)}}},
\end{align*}
but we also have
\begin{align*}
&\sum_{z \geq N^{2(1+\sigma+2 \epsilon)}} \tau(z) \sum_{ x \in \Z} \e^{- 2 \pi^2 N^2(x+ \alpha z + \frac{u}{\sqrt{2 \pi}N})^2- \frac{1}{8} \frac{z^2}{N^{2(1+\sigma+ \epsilon)}}} 
\\ &\qquad \ll_{\epsilon,\sigma}  \sum_{z \geq N^{2(1+\sigma+2\epsilon)}} z^{\frac{\epsilon}{2(1+\sigma+2 \epsilon)}}\e^{- \frac{1}{16} \frac{z^2}{N^{2(1+\sigma+ \epsilon)}}} \sum_{ x \in \Z} \e^{- 2 \pi^2 N^2(x+ \alpha z + \frac{u}{\sqrt{2 \pi}N})^2- \frac{1}{16} \frac{z^2}{N^{2(1+\sigma+ \epsilon)}}}
\\ &\qquad \ll_{\epsilon,\sigma}  \sum_{z \geq N^{2(1+\sigma+2\epsilon)}} \sum_{ x \in \Z} \e^{- 2 \pi^2 N^2(x+ \alpha z + \frac{u}{\sqrt{2 \pi}N})^2- \frac{1}{16} \frac{z^2}{N^{2(1+\sigma+ \epsilon)}}},
\end{align*}
where we use that for $z \geq N^{2(1+\sigma+2\epsilon)}$
\begin{align*}
z^{\frac{\epsilon}{2(1+\sigma+2 \epsilon)}}\e^{- \frac{1}{16} \frac{z^2}{N^{2(1+\sigma+ \epsilon)}}} \leq 1, 
\end{align*}
for all $N$ sufficiently large. Thus, we find that
\begin{align}
\begin{aligned}
\label{eq: First Lattice translation}
&\frac{1}{N} \sum_{1 \leq n \leq N^{\sigma+\epsilon}} \sum_{\bm v \in \Lambda_{\mathfrak z_{n,N}^{\alpha}}^*} \e^{-\pi \| \bm v+u \bm{e_1}\|^2}  
 \\ & \quad \ll_{\epsilon,\sigma} \frac{1}{N^{1-\epsilon}} \sum_{\bm v \in \Delta^\alpha_{N,(N^{1+\sigma+\epsilon})^{-1}}} \e^{-\frac{1}{16} \| \bm v + u \bm e_1 \|^2}.
 \end{aligned}
\end{align}
Let us apply \Cref{Siegel transform estimate} with $P=N, Q= (N^{1+\sigma+\epsilon})^{-1}, \zeta =N^{-1}, Z =N^{1+\sigma+\epsilon}$ to obtain the following estimate
\begin{align}
\label{eq: First Siegel}
\sum_{\bm v \in \Delta^\alpha_{N,(N^{1+\sigma+\epsilon})^{-1}}} \e^{-\frac{1}{16}\| \bm v + u \bm e_1 \|^2}  \ll \bigg(1+ \mathcal N_{\Delta^\alpha_{N^{(2+\sigma+\epsilon)/2}}}(\sqrt 2 N^\frac{\sigma+\epsilon}{2})\bigg).
\end{align}
Moreover, according to \Cref{Height estimate}, since $\alpha$ is Diophantine of type $\kappa$, we have
\begin{align*}
a(\Delta^\alpha_{N^{(2+\sigma+\epsilon)/2}})^{-1} \ll N^{\frac{2+\sigma+\epsilon}{2}(1- \frac{2}{\kappa})}.
\end{align*}
Hence, \Cref{unimodular Lipschitz principle} implies 
\begin{align}
\label{eq: First Lipschitz}
\mathcal N_{\Delta^\alpha_{N^{(2+\sigma+\epsilon)/2}}}(\sqrt 2 N^\frac{\sigma+\epsilon}{2}) \ll N^{\sigma+\epsilon}+N^{\frac{2+\sigma+\epsilon}{2}(1- \frac{2}{\kappa})+\frac{\sigma+\epsilon}{2}}
\end{align}
The assertion follows now from \eqref{eq: First Lattice translation},  \eqref{eq: First Siegel} and \eqref{eq: First Lipschitz}.
\end{proof}

\begin{proof}[Proof of \Cref{Theorem 1}]
According to \Cref{sufficient test functions} it suffices to show 
\begin{align*}
\lim_{N \to \infty} \mathrm R_2^{\sigma}(f,\psi \cdot \e^{-\| \, \cdot \, \|^2},\{\theta_{j,N}\},N) = \widehat f(0) \big(\mathcal F( {\psi \cdot \e^{-\| \, \cdot \, \|^2}})\big)(\bm 0),
\end{align*}
for any two test functions $f: \R \to \R$ and $\psi: \R^2 \to \R$  of class $\mathcal P$. According to \Cref{Reduction to Geometry of Numbers}, this last identity follows if
\begin{align*}
\lim_{N \to \infty} \mathrm E^\sigma_\epsilon(f,\psi,\{\alpha j^2 \}_j, N) = 0,
\end{align*}
for some fixed $\epsilon >0$. It is plain to see that
\begin{align*}
E^\sigma_\epsilon(f,\psi,\{\alpha j^2 \}_j, N) \ll_{\psi,f} \sup_{u \in \R}\sum_{b \in \{0,1\}} \frac{1}{N} \sum_{1 \leq n \leq N^{\sigma+\epsilon}} \sum_{\bm v \in \Lambda_{\mathfrak z_{n,N}^{\alpha},b}^*} \e^{-\pi \| \bm v+u \bm{e_1} \|^2}.
\end{align*}
and the right-hand side is bounded above by $N^{\sigma-1+\epsilon} + N^{\sigma-\frac{2+\sigma}{\kappa}+\epsilon(1-\frac{1}{\kappa})}$ by \Cref{First estimate}. We clearly require $\sigma <1$ and $\kappa < 1+\frac{2}{\sigma}$.
\end{proof}

\section{Upper Bound for $\sigma(\{ n^2 \alpha \}_n)$ for Almost Every $\alpha$}

\subsection{} This section contains the proof of \Cref{Theorem 2}. It follows essentially the same strategy as that of Rudnick and Sarnak \cite{rudnick-sarnak:1998} up to some small modifications.

\subsection{} The first step consists in showing that the \emph{variance of the $\sigma$-pair correlation} is small. Let $f: \R \to \R$ be a test function such that $\widehat f \in C_c(\R)$ has compact support and set
\begin{align*}
\mathrm R_2^\sigma(f,\{j^2 \alpha \}_j,N):= \mathrm R_2^\sigma(f,\mathbbm 1_{[1,N]^2},\{j^2 \alpha \}_j,N),
\end{align*}
as in \Cref{subsection:Functional}. Following Rudnick and Sarnak \cite{rudnick-sarnak:1998} we define
\begin{align*}
\mathrm X_N(\alpha):=  \mathrm R_2^\sigma(f,\{j^2 \alpha \}_j,N)-\widehat f(0).
\end{align*}

\begin{lemma}
\label{parseval}
For any $\epsilon >0$, we have
\begin{align*}
\| X_N \|_{\mathrm L^2([0,1])}^2  \ll_{f, \epsilon} \frac{1}{N^{2-\sigma-\epsilon}},
\end{align*}
where $\| \, \cdot \,\|_{\mathrm L^2([0,1])}$ denotes the standard norm on the $\mathrm L^2$-space on $[0,1]$ with respect to the Lebesgue measure.
\end{lemma}

\begin{proof}
Note that $\mathrm X_N$ is a periodic function in the variable $\alpha$. Thus, we can express it as a Fourier series
\begin{align*}
\mathrm X_N(\alpha) = \sum_{l \in \Z} c_l(N) \e(l \alpha),
\end{align*}
where
\begin{align*}
c_l = \frac{1}{N^2} \sum_{n \in \Z \setminus \{0 \}} \sum_{\substack{1 \leq j \neq k \leq N \\ (j^2 -k^2)n = l}} \widehat f \bigg( \frac{n}{N^\sigma} \bigg),
\end{align*}
and 
\begin{align*}
c_0 = \mathcal O(\frac{1}{N}), \text{ as } N \to \infty.
\end{align*}
Note  that $c_l(N) \ll \frac{\tau(|l|)^2}{N^2}$ for any $l \neq 0$, where $\tau$ denotes the divisor function, and hence for any fixed $\epsilon>0$
\begin{align*}
c_l(N) \ll_\epsilon \frac{|l|^\epsilon}{N^2} \text{ for any } l \neq 0.
\end{align*}
Moreover, for any fixed $\delta >0$ and all sufficiently large $N$
\begin{align*}
c_l(N) = 0  \text{ for all } l \geq N^{2+\sigma+\delta}.
\end{align*}
It follows from these two observations together with Parseval's identity that
\begin{align}
\| \mathrm X_N \|_{\mathrm L^2([0,1])}^2 \ll_{f,\epsilon} \frac{1}{N^{2-\sigma-\epsilon}}.
\end{align}
\end{proof}

\subsection{} The second step consists in noticing that small variance leads to almost everywhere convergence of the $\sigma$-pair correlation along a sparse subsequence of Planck constants. Fix $0<\delta<1$ such that
\begin{align}
\label{delta req 1}
(2 -\sigma)(1+\delta)>1
\end{align}
and choose $\epsilon>0$ small enough such that $(2-\sigma-\epsilon)(1+\delta)>1$. Let $\{N_m\}_m$ be a sequence of integers with
\begin{align*}
N_m \sim m^{1+\delta}.
\end{align*}
Then, according to \Cref{parseval}, we find
\begin{align*}
\infty > \sum_{m =1} ^\infty \frac{1}{m^{(2-\sigma-\epsilon)(1+\delta)}} \gg_{f,\epsilon} \sum_{m=1}^\infty  \| \mathrm X_{N_m} \|_{\mathrm L^2([0,1])}^2 = \int_0 ^1 \sum_{m=1}^\infty  |\mathrm X_{N_m}(\alpha) |^2 \, \d \alpha,
\end{align*}
and consequently
\begin{align}
\label{almost everywhere convergence}
\lim_{m \to \infty} \mathrm X_{N_m}(\alpha) = 0 \text{ for almost every } \alpha \in [0,1].
\end{align}
We can now choose a set $P(f)$ of full measure in $[0,1]$ such that for any $\alpha \in P(f)$, $\alpha$ is Diophantine and satisfies \eqref{almost everywhere convergence}. 

\subsection{} The final step consists in proving that for any $\alpha \in P(f)$, due to the Diophantine nature of $\alpha$, the oscillations $\mathrm X_{n}(\alpha)- \mathrm X_{N_m}(\alpha)$ along the sparse subsequence $\{N_m\}_m$ are small for all $N_m \leq n < N_{m+1}$.
As  $n-N_m \leq N_{m+1} - N_m \ll N_m^\delta$ it suffices to prove the following

\begin{lemma}
\label{oscillation decrease}
Let $0 \leq \sigma <2$ and  $\delta >0$ such that
\begin{align}
\begin{aligned}
\label{delta req 3}
&-2+\sigma+2\delta <0, \text{ and }\\
&-2+ \frac{1}{2} +\sigma +\delta <0,
\end{aligned}
\end{align} 
and let $\alpha$ be Diophantine. Then, 
\begin{align*}
\sup_{0 \leq l \leq N^{\delta}} |\mathrm X_{N+l}(\alpha)- \mathrm X_N(\alpha) | \to 0, \text{ as } N \to \infty.
\end{align*}
\end{lemma}

In view of \eqref{almost everywhere convergence} and \Cref{oscillation decrease}, it follows that
\begin{align*}
\lim_{N \to \infty} \mathrm X_N(\alpha) = 0,
\end{align*}
for all $\alpha \in P(f)$, as long as $\sigma$ and $\delta$ satisfy both \eqref{delta req 1} and \eqref{delta req 3}. It is plain to see that these two requirements are satisfied if
\begin{align*}
&0 \leq \sigma \leq 1 \text{ and } 0 < \delta < 1- \frac{\sigma}{2}, \text{ or } \\
&1 < \sigma <\frac{1}{4}(9 -\sqrt{17}) \text{ and } \frac{\sigma -1}{2 - \sigma} < \delta <\frac{3}{2} - \sigma.
\end{align*}
 \Cref{Theorem 2} follows then easily from this observation. The rest of this section will be devoted towards the proof of \Cref{oscillation decrease}.

\begin{proof}[Proof of \Cref{oscillation decrease}]  The number $\epsilon$ will denote a small positive quantity, whose value will be adapted as needed during the argument. Let us set $M= N^{\sigma+\epsilon}$. Then, for all sufficiently large $N$ we have
\begin{align*}
&\mathrm X_N(\alpha) = \frac{1}{N^2} \sum_{0 <|n| \leq M} \widehat f\bigg( \frac{n}{N^\sigma} \bigg) \sum_{1 \leq j \neq k \leq N} \e( n \alpha(j^2 -k^2)), \text{ as well as}\\
&\mathrm X_{N+l}(\alpha) = \frac{1}{(N+l)^2} \sum_{0 <|n| \leq M} \widehat f\bigg( \frac{n}{(N+l)^\sigma} \bigg) \sum_{1 \leq j \neq k \leq N} \e( n \alpha(j^2 -k^2)).
\end{align*} 
for all $0 \leq l \leq N^\delta$. Moreover, for any $0 \leq l \leq N^\delta$, we have
\begin{align*}
&\frac{1}{(N+l)^2} = \frac{1}{N^2}+ \mathcal O(N^{-3 +\delta}), \\
&\widehat f\bigg( \frac{n}{(N+l)^\sigma} \bigg)= \widehat f \bigg( \frac{n}{N^\sigma} \bigg) + \mathcal O (N^{-1+\delta+\epsilon}),
\end{align*}
where the last identity follows from $\frac{n}{(N+l)^\sigma}= \frac{n}{N^\sigma} + \mathcal O(\frac{MN^\delta}{N})$. Thus, it is easy to see that for $0 \leq l \leq N^\delta$ and $N$ sufficiently large,
\begin{align}
\begin{aligned}
\label{First oscillation estimate}
&\bigg |\mathrm X_{N+l}(\alpha)- \frac{1}{N^2} \sum_{0<|n| \leq M} \widehat f\bigg(\frac{n}{N^\sigma} \bigg) \sum_{1 \leq j \neq k \leq N+l} \e(n \alpha(j^2 -k^2)) \bigg| \\ & \ll \frac{1}{N^{3-\delta-\epsilon}} \sum_{0 < n \leq M} \bigg | \sum_{1 \leq j \neq k \leq N+l} \e(n \alpha(j^2 -k^2)) \bigg|
\\ & \ll N^{-2 +\sigma+\delta+2 \epsilon}+ \frac{1}{N^{3-\delta-\epsilon}} \sum_{0 < n \leq M} \big | S_\alpha(n,N+l) \big|^2,
\end{aligned}
\end{align}
where
\begin{align*}
S_\alpha(n,N) :=\sum_{1 \leq j \leq N} \e(n \alpha j^2).
\end{align*}
At this point let us note that we require
\begin{align}
\label{delta req 2}
-2 + \sigma+\delta<0.
\end{align}
\subsubsection{} 
Let us provide a short proof of the fact that
\begin{align}
\label{Weyl's inequality}
\sum_{n=1}^M |S_\alpha(n,N)|^2 \ll_\epsilon (NM)^{1+\epsilon}, \text{ for any } \epsilon >0,
\end{align}
for a Diophantine number $\alpha$. Indeed, it is easy to see that (e.g. Lemma 3.1 in \cite{davenport:2005})
\begin{align*}
|S_\alpha(n,N) |^2 \ll N + \sum_{v=1} ^N \min \big \{N, \frac{1}{\|2 \alpha n v\|_\Z} \big \},
\end{align*}
where $\| \, \cdot \, \|_\Z$ denotes the distance to the nearest integer.
Moreover,
\begin{align*}
\sum_{n =1}^M \sum_{v=1} ^N \min \big \{N, \frac{1}{\|2 \alpha n v\|_\Z} \big \} &\ll_\epsilon N^{\epsilon/2} \sum_{z= 1} ^{2 M N} \min \big \{N, \frac{1}{\|\alpha z\|_\Z} \big \}
\\ & =N^{1+\epsilon/2} \bigg(\sum_{\substack{1 \leq z \leq 2 MN \\ \|\alpha z \|_\Z \leq N^{-1} }} 1\bigg) + N^{\epsilon/2} \sum_{\substack{1 \leq z \leq 2 NM \\ \| \alpha z \|_\Z >N^{-1}} } \frac{1}{\| \alpha z \|_\Z}.
\end{align*}
First observe that 
\begin{align*}
\sum_{\substack{1 \leq z \leq 2 MN \\ \|\alpha z \|_\Z \leq N^{-1} }} 1 \leq \mathcal N_{\Delta_{\sqrt {2 M}N}}(2 \sqrt M) \ll M (MN)^{\epsilon/2},
\end{align*}
where we use the notation introduced in \Cref{subsection:lattice point count} and \Cref{subsection:diophantine lattices}, as well as \Cref{Height estimate}. Similarly, notice that
\begin{align*}
\sum_{\substack{1 \leq z \leq 2 NM \\ \| \alpha z \|_\Z >N^{-1}} } \frac{1}{\| \alpha z \|_\Z}  & \ll \sum_{r=0} ^{\log_2(N)} 2^r \sum_{\substack{1 \leq z \leq 2 NM \\ 2^{-(r+1)} \leq \| \alpha z \|_\Z \leq 2^{-r}}} 1\\ & \leq  \sum_{r=0} ^{\log_2(N)} 2^r \mathcal N_{\Delta^\alpha_{\sqrt {2^{r+1}MN}}}(\sqrt{2^{-r+2}MN}) 
\\ & \ll  (MN)^{1+\epsilon/2}.
\end{align*} 

\subsubsection{} Let us return to the proof  \Cref{oscillation decrease}. In view of \eqref{First oscillation estimate} and \eqref{Weyl's inequality} we find that
\begin{align}
\bigg |\mathrm X_{N+l}(\alpha)- \frac{1}{N^2} \sum_{0<|n| \leq M} \widehat f\bigg(\frac{n}{N^\sigma} \bigg) \sum_{1 \leq j \neq k \leq N+l} \e(n \alpha(j^2 -k^2)) \bigg| \ll_\epsilon N^{-2+\sigma+\delta +\epsilon}
\end{align}
for any $\epsilon>0$. In order to relate this last estimate to the oscillation $|\mathrm X_{N+l}(\alpha) - \mathrm X_N(\alpha)|$, notice that
\begin{align*}
&\bigg |\sum_{1 \leq j \neq k \leq N+l} \e(n \alpha(j^2 -k^2)) - \sum_{1 \leq j \neq k \leq N} \e(n \alpha(j^2 -k^2)) \bigg | 
\\ & \ll \bigg | \sum_{k=N+1} ^{N+l} \e(-n \alpha k^2) \bigg| \bigg | \sum_{j=1} ^{N} \e(-n \alpha j^2) \bigg| +  \bigg | \sum_{N+1 \leq j \neq k \leq N+l}\e(n \alpha(j^2 -k^2)) \bigg | 
\\& \ll l |S_\alpha(n,N)| + l^2.
\end{align*}
Thus,
\begin{align*}
|\mathrm X_{N+l}(\alpha) - \mathrm X_N(\alpha)| &\ll_{\epsilon,f} N^{-2+\sigma+\delta +\epsilon} + \frac{M l^2}{N^2} + \frac{l}{N^2} \sum_{n=1} ^M |S_\alpha(n,N)|
\\ &\ll N^{-2+\sigma+2\delta +\epsilon} + N^{-2+\frac{1}{2}+\sigma+\delta+\epsilon},
\end{align*}
where we use that $0 \leq l \leq N^\delta$ and apply the Cauchy-Schwarz inequality to the sum in order to utilize the estimate \eqref{Weyl's inequality}. Clearly, at this point we require
 $-2+\sigma+2\delta <0$ as well as $-2+\frac{1}{2}+\sigma+\delta<0$, where $0\leq \sigma <2$ and $\delta >0$. Finally, observe that the requirement \eqref{delta req 3} implies \eqref{delta req 2}.
\end{proof}

\printbibliography

\end{document}